\documentclass[12pt]{amsart}

\usepackage{amssymb,latexsym,amsmath,amsfonts}
\usepackage[mathscr]{eucal}
\usepackage{mathrsfs}
\usepackage{graphicx}
\usepackage{amscd}
\usepackage{amsthm}
\usepackage{amsxtra}
\usepackage{txfonts}
\usepackage{commath}
\usepackage[nointegrals]{wasysym}
\usepackage{bm}
\usepackage{float}
\usepackage[usenames,dvipsnames]{xcolor}
\usepackage{hyperref}
\hypersetup{colorlinks=true, citecolor=MidnightBlue, linkcolor=BrickRed}

\numberwithin{equation}{section}
\theoremstyle{definition}
\newtheorem{definition}{Definition}[section]
\theoremstyle{definition}
\newtheorem*{ntn}{Notation}

\newtheorem{conjecture}{Conjecture}[section]

\newtheorem{remark}[definition]{Remark}

\theoremstyle{plain}
\newtheorem{theorem}[definition]{Theorem}
\newtheorem*{thms}{Theorem}
\newtheorem{thmx}{Theorem}

\newtheorem{lemma}[definition]{Lemma}

\newtheorem{Prop}[definition]{Proposition}
\newcommand{\beas}{\begin{eqnarray*}}
\newcommand{\eeas}{\end{eqnarray*}}
\newcommand{\bes} {\begin{equation*}}
\newcommand{\ees} {\end{equation*}}
\newcommand{\be} {\begin{equation}}
\newcommand{\ee} {\end{equation}}
\newcommand{\bea} {\begin{eqnarray}}
\newcommand{\eea} {\end{eqnarray}}

\newcommand{\alp}{\alpha}
\newcommand{\bet}{\beta}
\newcommand{\de}{\delta}
\newcommand{\eps}{\varepsilon}
\newcommand{\gam}{\gamma}

\newcommand{\Lam}{\Lambda}
\newcommand{\kap}{\kappa}
\newcommand{\om}{\omega}

\newcommand{\sig}{\sigma}
\newcommand{\Sig}{\Sigma}

\newcommand\smpartl[2]{\frac{\partial{#1}}{\partial{#2}}}
\newcommand\partl[2]{\dfrac{\partial{#1}}{\partial{#2}}}

\newcommand{\bdy}{\partial}
\newcommand{\bdisc}{\bdy\Delta}
\newcommand{\cdisc}{\overline\Delta}
\newcommand{\Om}{\Omega}
\newcommand\opar[1]{D^{n-1}({#1})}
\newcommand\cpar[1]{\overline{D^{n-1}({#1})}}
\newcommand{\dom}{\overline\Delta\times D^{n-1}(t)}

\newcommand{\cont}{\mathcal{C}}
\newcommand\ckm[2]{\cont^{#1,#2}}

\newcommand\ockm[2]{\mathscr A^{#1,#2}}

\newcommand{\rea}{\operatorname{Re}}
\newcommand{\ima}{\operatorname{Im}} 
\newcommand\wt[1]{\widetilde{#1}}
\newcommand\wh[1]{\widehat{#1}}

\newcommand{\Sing}{\operatorname{Sing}}
\newcommand{\dist}{\operatorname{dist}}
\newcommand{\diam}{\operatorname{diam}}
\newcommand{\gr}{\operatorname{Graph}}
\newcommand{\bast}{\boldsymbol{*}}
\newcommand{\ann}{\text{Ann}}
\newcommand{\ba}{\boldsymbol a}

\newcommand{\bA}{\boldsymbol A}
\newcommand{\cA}{\mathcal{A}}

\newcommand{\B}{\mathcal{B}}

\newcommand{\bn}{\mathbf B^{n+1}}

\newcommand{\cD}{\mathcal{D}}
\newcommand{\fd}{\mathfrak d}
\newcommand{\cE}{\mathcal{E}}
\newcommand{\ev}{\operatorname{ev}}
\newcommand{\cF}{\mathcal{F}}
\newcommand{\fg}{\mathfrak g}
\newcommand{\cG}{\mathcal G}
\newcommand{\hol}{\mathcal{O}}
\newcommand{\fH}{\mathfrak H}
\newcommand{\cH}{\mathcal H}

\newcommand{\id}{\text{\bf I}}

\newcommand{\cJ}{\mathcal J}

\newcommand{\sL}{\mathscr{L}}
\newcommand{\cM}{\mathcal{M}}
\newcommand{\cN}{\mathcal N}
\newcommand{\cQ}{\mathcal Q}
\newcommand{\cP}{\mathcal{P}}

\newcommand{\cR}{\mathcal{R}}
\newcommand{\sn}{\mathbf S^n}
\newcommand{\bs}{\mathbf{s}}

\newcommand{\bt}{\mathbf{t}}
\newcommand{\tr}{\operatorname{T}}
\newcommand{\cU}{\mathcal{U}}

\newcommand{\cV}{\mathcal{V}}
\newcommand{\cW}{\mathcal{W}}

\newcommand{\zobar}{\overline{z_1}}
\newcommand{\zbar}{\overline{z}}
\newcommand{\CC}{\mathbb{C}^2}
\newcommand{\Cn}{\mathbb{C}^n}
\newcommand{\C}{\mathbb{C}} 

\newcommand{\rl}{\mathbb{R}}

\newcommand{\Rn}{\mathbb{R}^{n}}

\newcommand{\N}{\mathbb{N}}

\newcommand{\Gl}{\operatorname{GL}}
\title{Stability of the hull(s) of an $n$-sphere in $\Cn$}

\begin{document}
\author{Purvi Gupta}
\address{Department of Mathematics, Indian Institute of Science, 
Bangalore -- 560012, India}
\email{purvigupta@iisc.ac.in}
\author{Chloe Urbanski Wawrzyniak}
\address{Department of Mathematics, Rutgers University\\
New Brunswick, New Jersey 08854}
\email{ceu11@math.rutgers.edu}
\begin{abstract}
We study the (global) Bishop problem for small perturbations of $\sn$ --- the unit sphere of $\C\times\rl^{n-1}$ --- in $\Cn$. We show that if $S\subset\Cn$ is a sufficiently-small perturbation of $\sn$ (in the $\cont^3$-norm), then $S$ bounds an $(n+1)$-dimensional ball $M\subset\Cn$ that is foliated by analytic disks attached to $S$. Furthermore, if $S$ is either smooth or real analytic, then so is $M$ (upto its boundary). Finally, if $S$ is real analytic (and satisfies a mild condition), then $M$ is both the envelope of holomorphy and the polynomially convex hull of $S$. This generalizes the previously known case of $n=2$ (CR singularities are isolated) to higher dimensions (CR singularities are nonisolated).    
\end{abstract}
\maketitle

\section{Introduction}\label{sec_intro}

\subsection {Motivation and statement of result}
Given a compact subset $E\subset {\mathbb C}^n$, a classical problem in complex analysis is to compute its {\em envelope of holomorphy}, $\widetilde E$, which is the spectrum (i.e., maximal ideal space) of $\hol(E)$. Here, $\hol(E)$ denotes the algebra of germs of holomorphic functions defined in a neighborhood of $E$. In general, $\widetilde E$ need not be a subset of $\Cn$ in any natural way. However, if $E'$ is the intersection of all the 
pseudoconvex domains containing $E$, and the restriction map $\hol(E')\rightarrow\hol(E)$ is bijective, then $\wt E=E'$. From the point of view of applications, the problem of computing $\wt E$ is particularly important when $E$ is a smooth orientable submanifold of $\Cn$. For instance, a well-known result due to Harvey and Lawson (\cite{HaLa75}) says, among other things, that if a compact, orientable, maximally complex, CR submanifold $E\subset\Cn$ of dimension $2k+1$, $k>0$, is contained in the boundary of a strongly pseudoconvex domain, then it bounds a
complex analytic variety with isolated singularities, which serves as $\wt E$.  Since $\wt E$ has minimal volume
among all integral currents bound by $E$ (by the
Weingarten formula) it is the solution of a {\em complex plateau problem}. When $E\subset\Cn$ is a closed Jordan curve, i.e., $k=0$, although $E$ is {\em holomorphically convex} ($\wt E=E$), the complex plateau problem has implications for polynomial approximations on $E$. A classical result due to Wermer (\cite{We58}) describes the precise conditions under which $E$ bounds a complex variety $V$ in $\Cn$, in which case $E\cup V$ coincides with the polynomially convex hull of $E$. The {\em polynomially convex hull} of a compact set $E\subset\Cn$ is the set $\widehat E=\{z\in\Cn:|p(z)|\leq \sup_E|p|\ \text{for all polynomials $p$ on}\ \Cn\}$. In analogy with $\wt E$, we can view $\widehat E$ as the spectrum of $\mathcal P(E)$ --- the set of functions in $\cont(E)$ that are uniformly approximable by holomorphic polynomials on $E$. 
 
\vspace{8pt}

In the cases cited above, $E$ has odd
dimension and maximal CR dimension. In certain other cases, 
$\wt{E}$ is a Levi-flat (foliated by analytic varities) submanifold with $E$ as its boundary. In his seminal paper of 1965 (\cite{Bi65}), Bishop first
considered the local version of this problem: given a point $p$ on a submanifold $S\subset\Cn$ and a sufficiently small neighborhood $E\subset S$ of $p$, under what conditions will $\wt{E}$ be a smooth Levi-flat submanifold with $E$ as
part of its smooth boundary? This is known as the {\em Bishop problem}. The global version of this problem is typically studied for generic closed submanifolds of $\Cn$ admitting only nondegenerate CR singularities, i.e, points where the maximal complex tangent space has nongeneric dimension. These are the so-called {\em Bishop submanifolds}. If $S$ is totally real (at $p$), then $S$ is holomorphically convex (at $p$). Bishop discovered that for an $n$-manifold $S$
in ${\mathbb C}^n$, if $S$ at $p$ has a CR singularity of {\em elliptic} type, then $\wt S$ is nontrivial due to the presence of embedded complex analytic disks, or {\em Bishop disks}, attached to $S$ near $p$. Further, he conjectured that these disks foliate a unique Levi-flat submanifold that serves as $\wt E$ for a sufficiently small neighborhood $E\subset S$ of $p$, and contains $E$ as part of its
smooth (real analytic) boundary when $S$ is smooth (real analytic) near $p$. In the case of $\CC$, Bishop's conjecture was settled by Kenig-Webster (\cite{KeWe82}) in
the smooth category, and by Webster-Moser (\cite{MoWe83}), Moser (\cite{Mo85}) and Huang-Krantz (\cite{HuKr95}) in the real analytic category. For general $n$, Bishop's conjecture was finally shown to be true in the work of Huang (\cite{Hu98}), partially based on the previous work of Kenig-Webster (\cite{KeWe84}). In contrast to the ellipitc case, Forstneri{\v c}-Stout showed in \cite{FoSt91} that if $p$ is a {\em hyperbolic} complex point of a surface $S\subset\CC$, then the local envelope of holomorphy of $S$ at $p$ is trivial. 


The first major breakthrough for the global version of the Bishop problem
was made by Bedford and Gaveau, when
they proved in \cite{BeGa83} that any smooth real Bishop surface $S\subset {\mathbb
C}^2$ with only two elliptic complex tangent points (hence, $S$ is a sphere) and contained in
the boundary of a bounded strongly pseudoconvex domain bounds a
Levi-flat hypersurface $M$ that is foliated by embedded analytic disks attached to $S$. Moreover
$M$ is precisely $\wt S$ (and even $\widehat S$, in some cases). By the aforementioned work
of Kenig-Webster, Moser and Huang-Moser, $\wt S$ is smooth (real analytic) up to $S$ when $S$ is smooth (real analytic). Subsequently, this result was generalized by Bedford-Klingenberg in \cite{BeKl91} and Kruzhilin in \cite{Kr91}, who showed that $\wt S$ is a Levi-flat hypersurface bound by $S$, even when $S$ is a sphere admitting hyperbolic complex tangent points (of course, one cannot expect
full regularity of $\wt S$ at a hyperbolic complex tangent
point). This problem was then solved for topological spheres in $\CC$ by Shcherbina in \cite{Sh93} (also see \cite{ChSh95} for a more general result).

The study of Bishop disks has found many
important applications in the areas of classical dynamical systems,
symplectic geometry and topology (as seen in the work of Gromov, Hofer and Eliashberg; see \cite{Gr85}, \cite{Ho93} and \cite{El90}). However, the (global) Bishop problem for $n$-manifolds in $\Cn$ is only well-understood for spheres in dimension two. The following global version of the Bishop problem in general dimensions remains wide open:

\begin{conjecture}[\cite{Hu17}]\label{conj}
Let $S\subset {\mathbb C}^n$, $n\ge 3$, be
a smooth compact $n$-dimensional Bishop submanifold, with only nondegenerate
elliptic complex tangents. Suppose that $S$ is contained in the
boundary of a smoothly bounded strongly pseudoconvex domain. Then
$S$ bounds a smooth Levi-flat submanifold $M$, which has $S$
as its smooth boundary. Morever, $M$ serves as the envelope of holomorphy of $S$.
\end{conjecture} 

The purpose of the present paper is to make progress along  the
lines of Conjecture 1.1. We will consider the case when $S$ is a
small perturbation of the standard $n$-sphere in ${\mathbb C}^n$.
More precisely,  we study the hulls of small perturbations of the
following natural embedding of the $n$-sphere in $\Cn$. \bes
    \sn=\left\{(z,z')\in\C\times\C^{n-1}:|z|^2+||z'||^2=1,\ \ima z'=0\right\}.
\ees
Let $\bn$ denote the ball bound by $\sn$ in $\C\times\rl^{n-1}$,
and note that $\bn$ is both the envelope of holomorphy and the polynomially convex hull of
 $\sn$, and is trivially foliated by analytic disks.
  We establish the following stability result,
  which gives the first solution to Conjecture~\ref{conj} in a special
  case:

\begin{theorem}\label{thm_main} Let $\rho>0$ and $\de>0$. Then, there is an $\eps>0$ such that, for $k>>1$, if $\psi\in\cont^{3k+7}(\sn;\Cn)$ with $||\psi||_{\cont^3(\sn;\Cn)}<\eps$, then there is a $\cont^k$-smooth $(n+1)$-dimensional submanifold with boundary, $M\subset\Cn$, such that
\begin{enumerate}
\item $\bdy M=\Psi(\sn)$, where $\Psi=\id+\psi$ on $\sn$.
\item $M$ is foliated by an $(n-1)$-parameter family of embedded analytic disks attached to $\Psi(\sn)$.
\item There is a $\cont^k$-smooth diffeomorphism $j:\bn\rightarrow M$ with $||j-\id||_{\cont^2(\bn;\Cn)}<\de$.
\item If $\psi$ is $\cont^\infty$-smooth, then $M$ is $\cont^\infty$-smooth upto its boundary.
\item If $\psi$ is real analytic, then $M$ is real analytic upto its boundary.
\item If $\psi$ is real analytic and the complexified map $\psi_\C$ extends holomorphically to
    \bes
    \cN_\rho\sn_\C=\{\xi\in\C^{2n}:\dist(\xi,\sn_\C)<\rho\},
    \ees
where $\sn_\C=\{(z,\zbar)\in\C^{2n}:z\in\sn\}$, and $\sup_{\:\overline{\cN_r\sn_\C}}|\psi_\C|<\eps$, then $M=\widetilde{\Psi(\sn)}=\widehat{\Psi(\sn)}$.
\end{enumerate}
\end{theorem}

There has been important work on the complex plateau
problem in $\Cn$, $n\geq 3$, but when $S\subset\Cn$ is a real-codimension two Bishop submanifold with nonminimal CR points. In this setting, 
$S$ is expected to bound a Levi-flat hypersurface $M$. Here we
refer to the work Dolbeaut-Tomassini-Zaitsev (\cite{DoToZa05}, \cite{DoToZa11}) and Lebl-Noell-Ravisankar (\cite{LeNoRa20}) for the construction of $M$, and Huang-Yin (\cite{HuYi16}, \cite{HuYi17}), Valentin Burcea (\cite{Bu13a}, \cite{Bu13b}), and Fang-Huang (\cite{FaHu18}) for the regularity of $M$ at the CR points of $S$. The problem can also be formulated as a
 boundary value problem for a certain degenerate elliptic equation
 (called the Levi equation) and approached from a PDE point of view.
 However, it is hard to establish the foliated structure of the weak solutions obtained via this approach. The PDE approach has been carried out in the work of Slodkowski-Tomassini (\cite{SlTo91}).

We now describe the idea of our proof of Theorem 1.1. In order to
construct $M$, we first focus on the CR singularities of $\Psi(\sn)$.
The set of CR singularities of $\Psi(\sn)$ forms
an $(n-2)$-sphere consisting only of {\em nondegenerate elliptic}
CR singularities (see Lemma~\ref{lem_sing}). A point $p$ in an
$n$-manifold $S\subset\Cn$ is a nondegenerate elliptic
CR singularity of $S$ if, after a local holomorphic change of
coordinates, $S$ near $p=0$ is given by
    \beas
     z_n&=&|z_1|^2+2\lambda\rea(z_1^2)+O(|z|^3);\\
        y_j&=&O(|z|^3),\quad  j=2,...,n-1,
    \eeas
where $\lambda\in[0,\frac{1}{2})$. As discussed earlier, the local envelope of holomorphy of a smooth (real analytic) $S$ at such a $p$ is a smooth (real analytic)
$(n+1)$-dimensional manifold foliated by Bishop disks attached to $S$.

Away from its set of CR singularities, $\Psi(\sn)$ is totally real, so we must solve a Riemann-Hilbert problem to produce the necessary attached disks. Note that this technique has been used to establish stability results for attached disks by Bedford (\cite{Be81}) and Alexander (\cite{Al81}) for certain specific totally real submanifolds, and by Forstneri{\v c} (\cite{Fo87}) and Globevnik (\cite{Gl94}) for a more general class. Our setting (away from the singular set) coincides with that of Alexander's, and we use his construction to show that for any $k$ large enough, there is an $\eps_k>0$ such that every $\eps_k$-small $\cont^{k+2}$-perturbation of $\sn$ contains the boundary of a $\cont^{k}$-smooth manifold foliated by attached holomorphic disks. However, $\eps_k$ may shrink to zero as $k$ increases, and thus we need a different approach for $\cont^\infty$-smooth perturbations. For this, we fix a sufficiently small perturbation $\Psi$, construct the ($\cont^1$-smooth) foliation attached to $\Psi(\sn)$ {\` a} la Alexander, and then, use the Forstneri{\v c}-Globevnik multi-index theory for attached disks to smoothly reparametrize the foliation near each leaf.

Finally, to establish the polynomial convexity of $M$, we globally flatten $M$ to a domain in $\C\times\rl^{n-1}$, and use a trick due to Bedford for Levi-flat graphs of hypersurface type. In order to carry out this flattening, we must assume that our perturbation is real analytic with a uniformly bounded below radius of convergence on $\sn$. Hence, the assumptions stated in $(6)$ in Theorem~\ref{thm_main}. It is not clear whether these assumptions can be done away with.

\subsection {Plan of the paper} The proof of our main result is organized as follows. In Section~\ref{sec_prelim}, we collect some preliminary observations regarding the perturbations considered in this paper. In Section~\ref{sec_away}, we establish the stability of the holomorphic disks whose boundaries in $\sn$ lie outside a neighborhood of its CR singularities. This includes Alexander's construction of the disks (and, resulting foliation) for $\cont^3$-smooth perturbations (\S \ref{subsec_exist}), and the proof of the regularity of this foliation in the case of real analytic perturbations (\S\ref{sec_realan}) and $\cont^\infty$-perturbations (\S\ref{subsec_cksmooth}). Next, in Section~\ref{sec_patch}, we complete the proof of claims $(1)$ to $(5)$ in Theorem~\ref{thm_main} by patching up the construction in Section~\ref{sec_away} with the local hulls of holomorphy of the perturbed sphere nears its CR singularities. Finally, in Section~\ref{sec_final}, we establish the polynomial convexity of the constructed manifold under the stated assumptions.

\subsection{Acknowledgments} We are extremely grateful to Xiaojun Huang for his invaluable mathematical insights and comments on the subject of this paper. In particular, we thank him for suggesting the method of flattening which is crucially used in the proof of $(6)$ in our main theorem.
\section{Notation and other preliminaries}\label{sec_prelim}

\subsection{Notation and setup}\label{subsec_notn} We will use the following notation throughout this paper. 
\begin{itemize}
\item [$\bullet$] The unit disc and its boundary in $\C$ are denoted by $\Delta$ and $\bdy\Delta$, respectively.
\item [$\bullet$] The open Euclidean ball centered at the origin and of radius $r>0$ in $\rl^{k}$ is denoted by $D^{k}(r)$. 
\item [$\bullet$] Bold small letters such as $\bt$ and $\bs$ denote vectors in $\rl^{n-1}$. For the sake of convenience, we index the components of these vectors from $2$ to $n$, i.e., $\bt=(t_2,...,t_n)$. 
\item [$\bullet$] We will denote the identity map by $\id$, where the domain will depend on the context.
\item [$\bullet$] Given any	 normed function space $\left(\mathcal{F}(K),||\cdot||_{\mathcal F}\right)$ on a set $K\subset\Cn$, we let 
	\begin{itemize}
\item $\mathcal F(K;\rl)=\{f\in\mathcal{F}(K):f\ \text{is $\rl$-valued}\}$, with the same norm.  
\item $\mathcal F(K;\rl^n)=\{(f_1,...,f_n):K\rightarrow \rl^n:f_j\in\mathcal F(K;\rl)\}$, with $||(f_j)||_\mathcal F=\sup_j||f_j||_{\mathcal F}$.  
\item $\mathcal F(K;\Cn)=\{(f_1,...,f_n):K\rightarrow \Cn:f_j\in\mathcal F(K)\}$, with $||(f_j)||_\mathcal F=\sup_j||f_j||_{\mathcal F}$.
	\end{itemize}
\item For any $n$-dimensional submanifold $M\subset\Cn$, we denote the set of CR singularities of $M$ by $\Sing M$. 
\end{itemize}

We now make some preliminary observations on the perturbations considered in this article. Let $\B_{3}$ denote an $\eps$-neighborhood of the origin in $\cont^3(\sn;\Cn)$, where $\eps>0$ will be determined later on. We let $K_s=\{z\in\Cn:\operatorname{dist}(z,\sn)<s\}$, where $s>0$ is small enough so that there is a smooth retraction $\mathfrak r$ of $K_s$ to $\sn$. We may choose an $\eps>0$ small enough so that
\begin{itemize}
\item [$*$] there is a $t\in(0,s)$ such that for every $\psi\in\B_3$, the diffeomorphism $\Psi:K_s\rightarrow\Cn$ given by $z\mapsto z+\psi(\mathfrak r(z))$ satisfies $\Psi(\sn)\subset K_{t}\subset \Psi(K_s)$; and
\item [$*$] the map $\operatorname{Inv}:\B_3\rightarrow \cont^3(K_t;\Cn)$ given by $\psi\mapsto (\Psi^{-1}-\id)|_{K_t}$ is well-defined and $\cont^2$-smooth.  
\end{itemize}  
We denote $\Psi^{-1}|_{K_t}$ by $\Phi$ and $\operatorname{Inv}(\psi)=\Phi-\id$ by $\phi$. For $\phi\in\operatorname{Inv}(\B_3)$, we let 
	\bes
		\sn_\phi=(\id +\psi)(\sn),
	\ees
where the $\phi=\operatorname{Inv}(\psi)$. Thus, $z\in K=K_t$ satisfies $z\in \sn_\phi$ if and only if $z-\phi(z)\in\sn$. 

\subsection{On the structure of \texorpdfstring{$\text{\bf Sing}(\sn_\phi)$}{a}} Next, we study the structure of $\Sing\sn_\phi$ for $\cont^3$-small $\psi$. Note that $\Sing\sn=\sn\cap\{z_1=0\}$ is an $(n-2)$-dimensional sphere, at every point of which $\sn$ has a nondegenerate elliptic CR singularity. We claim that the same is true for $\Sing\sn_\phi$, for $\psi$ small enough. Since nondegenerate ellipticity of CR singular points is stable under $\cont^2$-small perturbations, it suffices to prove the following global result.

\begin{lemma}\label{lem_sing} Let $n\geq 3$. Given $\eta>0$, there is a $\tau\in(0,1)$ such that for any $\psi\in \tau\B_3$, there exists a $\cont^{2}$-smooth diffeomorphism $\iota:\sn\rightarrow\sn$ such that $(\Psi\circ\iota)(\Sing\sn)=\Sing\sn_\phi$, and $||\Psi\circ\iota-\id||_{\cont^2(\sn;\Cn)}<\eta$. In particular, $\Sing\sn_\phi$ is an $(n-2)$-dimensional sphere.
\end{lemma}
\begin{proof} We let $U^{n-1}$ denote the unit sphere in $\rl^n$. Consider the map $\Theta:D^2(1)\times U^{n-2}\rightarrow\Cn$ given by
	\bes
		\Theta:(a,b,s)
			\mapsto \left(a+ib,\sqrt{1-a^2-b^2}s+i0\right).
	\ees
Note that $\Theta$ parametrizes (and is invertible on) $\sn$ off the $z_1$-axis, and $\Theta^{-1}(\Sing\sn)=\{(0,s):s\in U^{n-2}\}$.

Now, let $R:\Cn\rightarrow\rl^n$ be given by $(z_1,z_2,...,z_n)\mapsto\left(|z_1|^2+\cdots+|z_n|^2-1,\ima(z_2),...,\ima(z_n)\right) $. We note that since $\Sing\sn=\{z\in\sn:\operatorname{rank}\operatorname{Jac}_\C R(z)<n\}$, and 
	\beas
	\operatorname{Jac}_\C(z_1,...,z_n)=
    \begin{pmatrix}
	   \overline{z_1} &  \overline{z_2} & \cdots &  \overline{z_n}\\
    	0		& \frac{1}{2i} & \cdots & 0\\ 
   	\vdots & \vdots & \ddots & \vdots\\
		0 & \cdots &\cdots & \frac{1}{2i}\\
		\end{pmatrix},
		\eeas
we have that $\Sing\sn=\{z\in\sn:\det\operatorname{Jac}_\C(z)=0\}$. We let $J:\B_3\times{D^2(1)}\times U^{n-2}\rightarrow\rl^2$ be given by $
		(\psi,a,b,s)\mapsto \det\operatorname{Jac}_\C (R\circ \Phi)(\Psi\circ \Theta(a,b,s))$, where $\Psi$ and $\Phi$ are related to $\psi$ as discussed above. Note that $J$ is a $\cont^2$-smooth map such that

	\begin{itemize}
	\item  $\Theta(a,b,s)\in\Psi^{-1}(\Sing\sn_\phi)$ if and only if $J(\psi,a,b,s)=0$ (after possibly shrinking $\B_3$); 
	\item For any $s\in U^{n-2}$, $J(0,0,s)=0$ and $D_{a,b}J(0,0,s)=(\frac{1}{2i})^{n-1}\begin{pmatrix}1&0\\0&-1\end{pmatrix}$.
	\end{itemize}

Thus, by the implicit function theorem (and the compactness of $U^{n-2}$), there is a $\tau\in(0,1)$, a neighborhood $V$ of $0$ in $\C$, and a $\cont^2$-smooth map $\Gamma:\tau\B_3\times U^{n-2}\rightarrow\C$ such that $J(\psi,z_1,s)=0$ if and only if $z_1=\Gamma(\psi,s)$, for any $(\psi,z_1,s)\in \tau\B_3\times V\times U^{n-2}$. 

Thus, in the parameter space ${D^2(1)}\times U^{n-2}$, $\Psi^{-1}(\Sing\sn_\phi)$ pulls back to the $\cont^2$-smooth graph $\mathscr G_\psi=(\Gamma(\psi,s),s)$. By shrinking $\tau$ further, we may assume that $\mathscr G_\psi$ lies in a thin neighborhood $N$ of $\mathscr G_0$. As both $\mathscr G_0$ and $\mathscr G_\psi$ are graphs over $U^{n-2}$, there is a diffeomorphism $\wt\iota$ of ${D^2(1)}\times U^{n-2}$ that is $\cont^2$-close to identity, maps $\mathscr G_0$ to $\mathscr G_\psi$ and is identity outside $N$. Setting $\iota=\Theta \circ\wt\iota\circ\Theta^{-1}$, we obtain the necessary map.
\end{proof}

\section{Away from the set of CR singularities}\label{sec_away}
 \subsection{Preliminaries}\label{subsec_prelim} In this section, we define some function spaces and maps that will be used throughout this construction. First, we recall some basic notions from infinite-dimensional analysis. Recall that a map $T:E \rightarrow F$ between Banach spaces is $k$-times continuously differentiable, or $\cont^k$-smooth, if it admits $k$ continuous Fr{\'e}chet derivatives. That is, for each $j=1,...,n$, there is a continuous map $D^jT$ from $E$ into $\sL^j(E,F)$ ---- the space of bounded $j$-linear maps from $E\oplus\cdots\oplus E$ ($j$ copies) to $F$ endowed with the standard topology --- satisfying
\bes
    \lim_{\norm{h}_E \rightarrow 0}
		\frac{\norm{D^{j-1}(x+h) - D^{j-1}(x) - D^{j}T(x)(h)}_{\sL^{j-1}(E,F)}}
			{\norm{h}_E} = 0.
\ees
If $E=E_1\oplus\cdots\oplus E_k$, then the partial Fr{\' e}chet derivatives $D_j T(x_1,...,x_k)$ are defined by analogy with partial derivatives from ordinary calculus. The map $T$ is said to be analytic at $a\in E$ if there is a $\rho>0$, and a sequence of maps $T^j\in\sL^j(E,F)$ with $\sum_{j\geq 0}||T_j||\:\rho^j<\infty $, such that $T(a+h)=T(a)+\sum_{j\geq 1} T_j(h,...,h)$, for $||h||_E<<1$.  Alternatively,  if $T$ is infinitely differentiable, it suffices to produce a neighborhood $V_a$ of $a$ and constants $c, \rho$ such that $\norm{D^jT} \le c j!\:\rho^{-j}$ on $V_a$. It follows that the composition or product of analytic maps is again analytic. Finally, of particular import for this paper is the implicit function theorem.
\begin{thms}[Implicit Function Theorem for Banach Spaces]
Let $E, F, G$ be Banach spaces, $T$ a continuously differentiable mapping of an open set $A$ of $E \times F$ into $G$. Let $(x_0, y_0) \in A$ be such that $T(x_0, y_0) = 0$, and $D_2T(x_0, y_0)$ is a linear homeomorphism of $F$ onto $G$. Then, there is an open neighborhood $U_0$ of $x_0$ in $E$ such that, for every open connected neighborhood $U$ of $x_0$, contained in $U_0$, there is a unique continuous mapping $u$ of $U$ into $F$ such that $u(x_0) = y_0$, $(x, u(x)) \in A$, and $T(x, u(x)) = 0$ for any $x \in U$. Furthermore, $u$ has the same regularity as $T$.
\end{thms}
Similarly, one has the inverse function theorem for Banach spaces. More details on this functional analysis background can be found in \cite{Di60}.

We now collect some functions spaces on the unit circle $\bdisc$. Given $0<\alp<1$, let 	
	\bes
		\ckm{0}{\alp}{}(\bdisc)=
			\left\{f\in\cont(\bdisc;\C):||f||_\alp
				=||f||_\infty+
					\sup_{\substack{x,y\in \bdisc\\x\neq y}}
						\frac{|f(x)-f(y)|}{|x-y|^\alp}<\infty
					\right\},
	\ees
where $||f||_\infty=\sup_{x\in \bdisc}||f(x)||$. For $k\in\N$, let
	\bes
		\ckm{k}{\alp}{}(\bdisc)=
			\left\{f\in\cont^k(\bdisc;\C)
				:||f||_{k,\alp}=\sum_{j=0}^k||D^j\!f||_\alp<\infty\right\}.
	\ees
Note that we use notation $\ckm{k}{\alp}(\bdisc;\rl)$, $\ckm{k}{\alp}(\bdisc;\rl^n)$ and $\ckm{k}{\alp}(\bdisc;\Cn)$ according to the convention established in Section~\ref{sec_prelim}. We will use the notation $B_{k,\alp}(f,r)$ to denote the ball of radius $r$ centered at $f$ in the Banach space $\ckm{k}{\alp}(\bdisc)$ (or in $\ckm{k}{\alp}(\bdisc;\Cn)$, depending on the context).

We also work with the Banach space
	\be\label{eq_holhol}
		\ockm{k}{\alp}(\bdisc)=\{f\in\ckm{k}{\alp}(\bdisc):
		\exists\ \wt f\in\hol(\Delta)\cap\ckm{k}{\alp}(\cdisc)\ 
		\text{such that}\ \wt f|_{\bdisc}=f\}
	\ee
with the same norm as that on $\ckm{k}{\alp}(\bdisc)$. It is known that if $f$ and $\wt f$ are as above, then $||\wt f||_{\ckm{k}{\alp}(\cdisc)}\lesssim||f||_{k,\alp}$. The following lemma will prove useful later.

\begin{lemma}\label{lem_ev}
For any $k\in\N$, the map $\ev:\cdisc\times\ockm{k}{\alp}{}(\bdisc;\Cn)\rightarrow\Cn$ given by $\ev(\xi,f)=\wt f(\xi)$ is $\cont^k$-smooth on $\cdisc\times\ockm{k}{\alp}{}(\bdisc;\Cn)$ and real-analytic on $\Delta\times\ockm{k}{\alp}{}(\bdisc;\Cn)$. 
\end{lemma}
\begin{proof} We note that $f\mapsto\wt f$ is a bounded linear transformation. Now, we have that
\bes
 D^j \ev(\xi,f)(\zeta_1,h_1)\cdots(\zeta_j,h_j)
		=\wt f^{(j)}(\xi)\zeta_1\cdots\zeta_j+\sum_{\ell=1}^j
		 \wt {h_\ell}^{(j-1)}(\xi)\frac{\zeta_1\cdots\zeta_j}{\zeta_\ell}.
\ees
Since all the derivatives of $f$ up to order $k$ satisfy a H{\" o}lder condition of the form 
	\bes
		|\wt f^{(j)}(\xi_1)-\wt f^{(j)}(\xi_2)|\leq ||f||_{k,\alp}|\xi_1-\xi_2|^\alp,
			\qquad \xi_1,\xi_2\in\cdisc,
	\ees
the continuity of $D^je$ for $j\leq k$ follows. Thus, we obtain the first part of the claim.

Next, we observe that for any $(\xi,f)\in\Delta\times\ockm{k}{\alp}(\bdisc;\Cn)$, we may write 
	\bes
		\ev\big((\xi,f)+(\zeta,h)\big)=\ev(\xi,f)
			+\sum_{j\geq 1}A_j(\underbrace{(\zeta,h)\cdots (\zeta,h)}
				_{j\ \text{times}})
	\ees
whenever $f,h\in \ockm{k}{\alp}(\bdisc;\Cn)$ and $|\zeta-\xi|<1-|\xi|$, where $A_j$ is the symmetric $j$-linear map 
	\bes
		\big((\zeta_1,h_1),...,(\zeta_j,h_j)\big)
			\mapsto \frac{\wt f^{(j)}(\xi)}{j!}\zeta_1\cdots\zeta_j
				+\sum_{\ell=1}^j\frac{\wt h_\ell^{(j-1)}(\xi)}{j!}
					\frac{\zeta_1\cdots\zeta_j}{\zeta_\ell}.
	\ees
By Cauchy's estimates, we have that $||A_j||\leq (1+||f||_{k,\alp})$, $j\in\N$. Thus, $\sum_{j\in\N}||A_j||r^{j}<\infty$ for any $r<1$, which establishes the real-analyticity of $\ev$ at $(\xi,f)$.
\end{proof}

\begin{remark}\label{rem_int}
Here onwards, we will identify $f$ and $\wt f$, i.e., for $f\in\ockm{k}{\alp}{}(\bdisc;\Cn)$ and $\xi\in \Delta$, we will denote $\wt f(\xi)$ simply by $f(\xi)$. 
\end{remark}

Next, given $f\in\ckm{k}{\alp}(\bdisc;\rl)$, we let $\cH(f)$ be given by
	\be\label{eq_hilb}
		f=a_0+\sum_{n=1}^\infty a_ne^{in\theta}+\overline{a_n}e^{-in\theta}
		\mapsto 
		\cH(f)=\sum_{n=1}^\infty -ia_ne^{in\theta}+i\overline{a_n}e^{-in\theta}
	\ee
Note that $\cH$ is the {\em standard Hilbert transform}. It is well known that $\cH$ is a bounded linear transformation from $\ckm{k}{\alp}(\bdisc;\rl)$ to itself. We then define $\cJ:\ckm{k}{\alp}(\bdisc;\rl)\rightarrow \ckm{k}{\alp}(\bdisc)$ as
	\bes
		\cJ:f\mapsto f+i\cH(f).
	\ees
Clearly, $\cJ$ is also a bounded linear transformation with $\cJ(\ckm{k}{\alp}(\bdisc))\subset\ockm{k}{\alp}(\bdisc)$.  Note that if $f$ is as in \eqref{eq_hilb}, then $\cJ(f)(0)=a_0$. In an abuse of notation, the component-wise application of $\cH$ and $\cJ$ on elements in $\ckm{k}{\alp}(\bdisc;\rl^n)$ is also denoted by $\cH$ and $\cJ$, respectively.  

Lastly, we fix a parametrization for the holomorphic discs that foliate the hull of $\sn$. For any $(\xi,\bt)\in\cdisc\times\opar{1}$, we let $
		\fg_\bt(\xi)=\left(\sqrt{1-||\bt||^2}\xi,\bt\right)$. The perturbed sphere will be shown to be foliated by boundaries of discs that are perturbations of $\fg_\bt$. As discussed in Remark~\ref{rem_int}, we also use $\fg_\bt$ to denote $\fg_\bt|_{\bdisc}$.

\subsection{Existence of the foliation}\label{subsec_exist} In this section, we follow Alexander's approach (see \cite{Al81}) to construct a $\cont^1$-smooth manifold $M_{\text{TR}}\subset\Cn$ that is foliated by holomorphic discs whose boundaries are attached to the totally real part of $\sn_\phi$. For this, we first solve the following nonlinear Riemann-Hilbert problem: find a function $f:\cdisc\rightarrow\C$ that is holomorphic on $\Delta$ and whose boundary values on $\bdisc$ satisfy $|f(z)-\gamma(z)|=\sig(z)$, where $\gam(z)$ is close to $0$ (in some appropriate norm) and $\sigma$ is a positive function on $\bdy\Delta$. The solutions to the above problem give analytic discs attached to the torus $\{|z_1|=1,|z_2-\gam(z_1)|=\sig(z_1)\}$ in $\CC$. 

\begin{lemma}\label{lem_disc1d} Let $\alp\in(0,1)$. There is an open set $\Om\subset\ckm{1}{\alp}(\bdisc)\oplus\ckm{1}{\alp}(\bdisc;\rl)$ such that
	\bes
		\{(0,\sig):\sig>0\}\subset\Om\subset\{(\gam,\sig):\sig>0\},
	\ees 
and there is an analytic map $E:\Om\rightarrow\ockm{1}{\alp}(\bdisc)$ such that 
	\begin{itemize}
		\item [$(i)$] if $(\gam,\sig)\in\Om$ and $E(\gam,\sig)=f$, then $|f-\gam|=\sig$ on $\bdisc$, $f(0)=0$, and $f'(0)>0$;
		\item [$(ii)$] $E(0,c)(\xi)\equiv c\,\xi$ for $\xi\in \bdisc$, when $c$ is a positive constant function.
	\end{itemize}
\end{lemma}

\begin{proof} 
The idea of the proof is as follows. Given $(\gam,\sig)\in\ckm{1}{\alp}(\bdisc)\oplus\ckm{1}{\alp}(\bdisc;\rl)$ with $\sigma>0$, if there is an $\eta\in\ckm{1}{\alp}(\bdisc)$ that satisfies
	\be\label{eq_alpbet}
		\gam=\eta e^{\cJ(\log\sig)-\cJ(\log|\mathfrak{g}-\eta|)},
	\ee
where $\fg(\xi)=\xi$, $\xi\in \bdisc$, and $\cJ:\ckm{1}{\alp}(\bdisc;\rl)\rightarrow\ockm{1}{\alp}(\bdisc)$ is the operator defined in Section~\ref{subsec_prelim}, then, setting $E(\gam,\sig)=f=\fg e^{\cJ(\log\sig)}e^{-\cJ(\log|\fg-\eta|)}$, we have that
	\be\label{eq_soln}
		|f-\gam|
		=\left|\fg e^{\cJ(\log\sig)}e^{-\cJ(\log|\fg-\eta|)}
			-\eta e^{\cJ(\log\sig)}e^{-\cJ(\log|\fg-\eta|)}\right|
		=e^{\log\sig}|\fg-\eta|e^{-\log|\fg-\eta|}=\sig.
	\ee
Moreover, $f(0)=0$ and $f'(0)=e^{(J\log(\sigma/|\fg-\eta|))(0)}>0$. So, we must solve for $\eta$ in \eqref{eq_alpbet} for $(\gam,\sig)$ close to $(0,\sig)$ when $\sig>0$. But any solution of \eqref{eq_alpbet} corresponding to $(\gam,\sig)$ is also a solution corresponding to $(\gam e^{-\cJ(\log\sig)},1)$. Thus, it suffices to establish the solvability of \eqref{eq_alpbet} near $(0,1)\in \ckm{1}{\alp}(\bdisc)\oplus\ckm{1}{\alp}(\bdisc;\rl)$.

Let $U=\{\eta\in\ckm{1}{\alp}(\bdisc):||\eta||_\infty<1\}$, which is an open set in $\ckm{1}{\alp}{}(\bdisc)$. For $\eta\in U$, let $A(\eta)=e^{-\cJ(\log|\fg-\eta|)}$. We claim that 
	\be\label{eq_claimA}
		\text{$A:U\rightarrow\ockm{1}{\alp}{}(\bdisc)$ is an analtyic map with $A(0)=1$.}
	\ee
Further, letting $Q(\eta)=\eta\cdot A(\eta)$, we claim that
	\be\label{eq_claimQ}
	\text{ $Q:U\rightarrow \ckm{1}{\alp}{}(\bdisc)$ is an analytic map with
			 $Q(0)=0$ and $Q'(0)=\id$.}
	\ee 
Assuming \eqref{eq_claimA} and \eqref{eq_claimQ} for now, we can apply the inverse function theorem for Banach spaces to $Q$ to obtain open neighborhoods $\cU\subseteq U$ and $V$ of $0$ in $\ckm{1}{\alp}{}(\bdisc)$ such that $Q$ is an analtyic diffeomorphism from $\cU$ onto $V$. Set 
	\bes
		\Om=\{(\gam,\sig)\in
			\ckm{1}{\alp}(\bdisc)\oplus\ckm{1}{\alp}(\bdisc;\rl)
				:\sigma>0\ \text{and}\ \gam e^{-\cJ(\log\sig)}\in {V}\}
	\ees
and observe that $\eta=Q^{-1}(\gam e^{- \cJ(\log\sig)})$ solves \eqref{eq_alpbet} for every $(\gam,\sig)\in\Om$.

Now set $E_{\pm}:\ckm{1}{\alp}(\bdisc;\rl_{>0})\rightarrow \ockm{1}{\alp}{}(\bdisc)$ by $E_{\pm}(\sig)=e^{\pm \cJ(\log\sig)}$. The proof of \eqref{eq_claimA} below can be imitated to check that $E_{\pm}$ are analytic maps. Further, $M_\fg:\ockm{1}{\alp}{}(\bdisc)\rightarrow\ockm{1}{\alp}{}(\bdisc)$ defined by $M_\fg(h)=\fg h$ is also analytic since it is a bounded linear transformation. Thus, the map $E:\Om\rightarrow\ockm{1}{\alp}{}(\bdisc)$ given by
	\beas
		E(\gam,\sig)&=& 
	E_{+}(\sig) \left(M_\fg \circ A\circ Q^{-1}\right)(\gam E_{-}(\sig))
	\eeas
is analytic. As shown in \eqref{eq_soln}, it satisfies $(i)$. Also, $E(0,c)=E_+(c) M_\fg(1)=c\fg$, for $c>0$. 

We must now prove \eqref{eq_claimA} and \eqref{eq_claimQ}. For $\eqref{eq_claimA}$, we first consider the map $L:\eta\mapsto\log|\fg-\eta|$. We use the fact that if $f\in\ckm{1}{\alp}{}(\bdisc)$ and $g\in\cont^2(f(\bdisc))$, then $g\circ f\in\ckm{1}{\alp}{}(\bdisc)$. We apply this fact to $f=\fg-\eta$ for $\eta\in U$, and $g(\cdot)=\log(|\cdot|)$ to obtain that $L(U)\subset\ckm{1}{\alp}{}(\bdisc;\rl)$. Now, for a fixed $\eta\in U$ and a small $h\in\ckm{1}{\alp}{}(\bdisc)$, we have that
	\beas
		L(\eta+h)-L(\eta)&=&\log|\fg-\eta-h|-\log|\fg-\eta|
			=\log\left|1-\frac{h}{\fg-\eta}\right|\\
		&=&\frac{1}{2}\log\left(1-\frac{h}{\fg-\eta}\right)+
			\frac{1}{2}\log\left(1-\frac{\bar h}{\bar {\fg}-\bar \eta}\right)\\
		&=&\frac{1}{2}\left(-2\rea\left(\frac{h}{\fg-\eta}\right)
				+O\big(||h||_{1,\alp}^2\big)\right)\qquad \text{as}
					\ ||h||_{1,\alp}\rightarrow 0,
	\eeas
where we are using the Taylor series expansion of $\log(1-z)$ and the submultiplicative property of $||\cdot||_{1,\alp}$ in the last step. Thus, $L$ is differentiable at $\eta$ and $DL(\eta)(h)=-\rea\left(\frac{h}{\fg-\eta}\right)$. Continuing in this way, we obtain that $D^{j}\!L:U\rightarrow \sL^j(\ckm{1}{\alp}{}(\bdisc),\ckm{1}{\alp}(\bdisc;\rl))$ exists and is given by $D^{j}\!L(\eta)(h_1,...,h_j)=-(j-1)!\rea\left(\frac{h_1\cdots h_j}{(\fg-\eta)^j}\right)$. Thus, for any $j\geq 1$, $D^{j}\!L$ is continuous on $U$ when $\sL^j(\ckm{1}{\alp}{}(\bdisc),\ckm{1}{\alp}(\bdisc;\rl))$ is given the standard norm topology. Finally, observe that 

\begin{equation}
    \norm{D^jL(\eta)} = \sup_{\norm{(h_1,\ldots, h_j)} \le 1}\norm{D^jL(\eta)(h_1, \ldots, h_j)} = (j-1)!\norm{\rea\left(\frac{h_1\cdots h_j}{(\fg-\eta)^j}\right)}  \le \frac{j!}{\norm{\fg - \eta}^j}
\end{equation}

Hence, $L$ is analytic. Now, the maps $\cJ$ and $u\mapsto e^{-u}$ are both analytic on $\ckm{1}{\alp}(\bdisc;\rl)$, since the former is a bounded linear transformation, and the latter has continuous derivatives of all orders of the following form $(h_1,...,h_j)\mapsto e^{-u}h_1\cdots h_j$ at any $u\in\ckm{1}{\alp}(\bdisc;\rl)$. Thus, $A$ being the composition of analytic maps, is itself analytic. Further, as $L(0)=\log|\fg|=0$, $A(0)=1$. 

Now, recall that $Q(\eta)=\eta\cdot A(\eta)$. So, $Q(0)=0$. Being the product of two analytic, $Q$ is analytic at any $\eta\in U$. Now, since $DQ(\eta)(h)= \eta DA(\eta)(h)+hA(\eta)$, we have that $DQ(0)(h)\equiv h$. This gives \eqref{eq_claimQ} and concludes our proof. 
\end{proof}

We now apply Lemma~\ref{lem_disc1d} to solve a nonlinear Riemann-Hilbert problem in $n$ functions. Note that the same problem will be solved using a different technique in Section~\ref{subsec_cksmooth}, where we will improve the regularity of the manifold constructed here. 

\begin{lemma}\label{lem_discs} Let $\alp\in(0,1)$. There is an open neighborhood $\wt\Om$ of $D^{n-1}(1)\times \{0\}$ in $D^{n-1}(1)\times\cont^3(K;\Cn)$ and a $\cont^1$-smooth map $F:\wt\Om\rightarrow \ockm{1}{\alp}{}(\bdisc;\Cn)$ such that $F(\bt,0)=\fg_\bt$, and if $F(\bt,\phi)=f=(f_1,...,f_n)$ for $(\bt,\phi)\in\wt\Om$, then $f(\bdisc)\subset \sn_\phi$, $f_1(0)=0$ and $f_1'(0)>0$.
\end{lemma}
\begin{proof} Recall that from Lemma~\ref{lem_disc1d}, there exists an open set $\Om\subset\ckm{1}{\alp}(\bdisc)\oplus\ckm{1}{\alp}(\bdisc;\rl)$ so that the solution operator $E$ is smoothly defined on $\Om$. 
Now, for $(\bt,\phi,f)\in D^{n-1}(1)\times\cont^3(K;\Cn)\times\ockm{1}{\alp}{}(\bdisc;\Cn)$, consider the map
	\bes
		P:(\bt,\phi,f)\mapsto\left(\phi_1(f),\sqrt{1-\Sig(\bt,\phi,f)}\right),
	\ees
where $\Sig(\bt,\phi,f)=\sum_{j=2}^n\left(t_j+H(\ima\phi_j(f))-\rea\phi_j(f)\right)^2$.
Then, $P$ is a $\cont^1$-smooth map from $W$ into $\ckm{1}{\alp}{}(\bdisc)\oplus\ckm{1}{\alp}(\bdisc;\rl)$, where $W=\{(\bt,\phi,f):f(\bdisc)\subset K\ \text{and}\ |\Sig(\bt,\phi,f)(\xi)|<1\ \text{for all}\ \xi\in \bdisc\}$. This is a consequence of the following observations. 		
\begin{enumerate}
    \item $P$ is clearly $\cont^\infty$-smooth in the $\bt$ variable.
    \item Since $H$ and $f\mapsto f^2$ are $\cont^\infty$-smooth from $\ckm{1}{\alp}{}(\bdisc)$ to $\ckm{1}{\alp}{}(\bdisc)$, and $f\mapsto \sqrt{f}$ is $\cont^\infty$-smooth from $\ckm{1}{\alp}(\bdisc;\rl_{>0})$ to $\ckm{1}{\alp}(\bdisc;\rl)$, our claim reduces to $(3)$ below. 
    \item If $\om=\{(\varphi,f)\subset\cont^{3}(B)\times\ckm{1}{\alp}{}(\bdisc):f(\bdisc)\subset \text{dom}(\varphi)\}$, where $B\subset \C$ is some closed ball, then the map $(\varphi,f)\mapsto \varphi(f)$ is $\cont^1$-smooth from $(\om,||\cdot||_{3}\oplus||\cdot||_{1,\alp})$ to $(\ckm{1}{\alp}{}(\bdisc),||\cdot||_{1,\alp})$. 
\end{enumerate}

Next, we note that when $\bt\in D^{n-1}(1)$, $(\bt,0,\fg_\bt)\in W$ and $P(\bt,0,\fg_\bt)=(0,\sqrt{1-||\bt||^2})\in\Om$.  So, there exists an open set $\cW\subset \rl^{n-1}\oplus \cont^3(K;\Cn)\oplus\ockm{1}{\alp}(\bdisc;\Cn)$ such that
	\begin{enumerate}
		\item [$(i)$] $(\bt,0,\fg_\bt)\in \cW$ for all $\bt\in D^{n-1}(1)$, 
		\item [$(ii)$] $\cW\subseteq W$,
		\item [$(iii)$] $P(\cW)\subseteq \Om$.
\end{enumerate} 

Now, consider the map $R:\cW\mapsto \ockm{1}{\alp}{}(\bdisc;\Cn)$ given by
	\be\label{eq_finalR}
		R(\bt,\phi,f)
		=f-\left(E\circ P(\bt,\phi,f),
			\bt+H(\ima \boldsymbol \phi(f))+i\ima \boldsymbol\phi (f)\right),
	\ee
where $\boldsymbol\phi$ denotes the tuple $(\phi_2,...,\phi_n)$, and $H$ acts component-wise. The map $R$ is $\cont^1$-smooth. Note that $R(\bt,0,\fg_\bt)=0$ and $D_3R(\bt,0,\fg_\bt)=\id$ on $\ockm{1}{\alp}{}(\bdisc;\Cn)$ for all $\bt\in D^{n-1}(1)$. So, by the implicit function theorem for Banach spaces, for each $\bt\in D^{n-1}(t)$, there exist neighborhoods $U_\bt$ of $\bt$ in $D^{n-1}(1)$, $V_\bt$ of $0$ in $\cont^3(K;\Cn)$ and $W_\bt$ of $\fg_\bt$ in $\ockm{1}{\alp}(\bdisc;\Cn)$, and a $\cont^1$-smooth map $F_\bt:U_\bt\times V_\bt\rightarrow W_\bt$ such that $F_\bt(\bt,0)=\fg_{\bt}$ and 
	\be\label{eq_uniq1}
	R(\bs,\phi,f)=0\ \text{for}\ (\bs,\phi,f)\in U_\bt\times V_\bt\times W_\bt\
		 \text{if and only if}\ 
		f=F_\bt(\bs,\phi).
	\ee
But, by uniqueness $F_{\bt_1}=F_{\bt_2}$ whenever the domains overlap. Thus, there exists an open set $\wt\Om\subset D^{n-1}(1)\times\cont^3(K;\Cn)$ such that $D^{n-1}(1)\times\{0\}\subset \wt\Om$, and a $\cont^1$-smooth map $F:\wt\Om\rightarrow\ockm{1}{\alp}{}(\bdisc;\Cn)$ such that $F(\bt,0)=\fg_\bt$ and $R(\bt,\phi,F(\bt,\phi))=0$ for all $(\bt,\phi)\in\wt\Om$.
The latter condition means that if $F(\bt,\phi)=f$, then 
    \bea
    |f_1-\phi_1(f)|^2+\sum_{j=2}^n(\rea f_j-\rea \phi_j(f))^2=1,\notag\\
    \ima(f_j)=\ima\phi_j(f),\quad j=2,...,n.\label{eq_attach}
    \eea
    In other words, $f(\bdisc)\subset\sn_\phi$. Further, from $(i)$ in Lemma~\ref{lem_disc1d}, $f_1(0)=0$ and $f_1'(0)>0$.
        \end{proof}

We are now ready to construct the manifold $M_{\text{TR}}$. 
\begin{theorem}\label{thm_diskconst} Given $t\in(0,1)$, there is a neighborhood $N_{t}$ of $0$ in $\cont^3(K;\Cn)$ such that $\overline{D^{n-1}(t)}\times N_{t}\subset \wt\Om$, and for $\phi\in N_{t}$, the map $\cF_\phi:\cdisc\times\opar{t}\rightarrow\Cn$ defined by 
	\bes
		\cF_\phi(\xi,\bt)=F(\bt,\phi)(\xi)
	\ees
is a $\cont^1$-smooth embedding into $\Cn$, with the the image set $M_{\text{TR}}=\cF_\phi(\cdisc\times\opar{t})$ a disjoint union of analytic discs with boundaries in $\sn_\phi$. Further, the map $\phi\mapsto \cF_\phi$ is a continuous map from $N_t$ into $\cont^1(\cdisc\times\opar{t};\Cn)$. 
\end{theorem}
\begin{proof} In Lemma~\ref{lem_discs}, the open set $\wt\Om\subset D^{n-1}(1)\times\cont^3(K;\Cn)$ contains $D^{n-1}(1)\times\{0\}$. Thus, by compactness, for any $t\in(0,1)$, there is an open neighborhood $N_t$ of $0$ in $\cont^3(K;\Cn)$ such that $\overline {D^{n-1}(t)}\times N_t\subset \wt\Om$. 

Now, for a fixed $\phi\in N_t$, note that $\cF_\phi$ is the composition of two $\cont^1$-smooth maps: 
	\beas
		(\xi,\bt)\mapsto (\xi,F(\bt,\phi));\ 
		(\xi, f)\mapsto \wt f(\xi).
	\eeas
The smoothness of the second map was established in Lemma~\ref{lem_ev}. Thus, $\cF_\phi$ is a $\cont^1$-smooth map. Since, for $\phi\in N_t$,  $\phi\mapsto F(\bt,\phi)$ is a $\cont^1$-smooth map, we have that $D \cF_\phi$ depends continuously on $\phi$. Quantitatively, this says that for some $C>0$, 
    \bes
        ||\cF_{\phi^1} -\cF_{\phi^2}||_1\leq C||\phi^1-\phi^2||_3
        \ees
for $\phi^1, \phi^2\in N_t$. Thus, shrinking $N_t$ if necessary, we have that $\cF_\phi$ is an embedding for all $\phi\in N_t$, since $\cF_0$ is an embedding.  
\end{proof}

\begin{remark}\label{rem_norm} Based on the above results, we call an $f=(f_1,...,f_n)\in\ockm{k}{\alp}(\bdisc;\Cn)$ a {\em normalized analytic disc attached to} $\sn_\phi$ if $f(\bdisc)\subset\sn_\phi$, $f_1(0)=0$ and $f_1'(0)>0$. Note that in the construction above, each $F(\bt,\phi)$ is a normalized analytic disc attached to $\sn_\phi$. 
\end{remark}

\subsection{Regularity of the foliation for real-analytic perturbations}\label{sec_realan} In this section, we will show that the manifold $M_{\text{TR}}$ constructed in Theorem~\ref{thm_diskconst} is, in fact, real-analytic if $\Psi$ is a real-analytic perturbation of $\sn$.

Let $\cW\subset \opar{1}\oplus\cont^3(K;\Cn)\oplus \ockm{1}{\alp}(\bdisc;\Cn)$, $R$ and $F$ be as in the previous section (see \eqref{eq_finalR}). Recall that $R$ is a $\cont^1$-smooth map and $D_3R(\bt,0,\fg_\bt)=\id$ on $\ockm{1}{\alp}(\bdisc;\Cn)$, for all $\bt\in\opar{1}$. Thus, given $t\in(0,1)$, there is an $\eps_t>0$ such that, if $||\phi||_{\cont^3}<\eps_t$, then
	\begin{itemize}
		\item [$\bullet$] $\phi\in N_t$ where $N_t\subset\cont^3(K;\Cn)$ is a neighborhood of $0$ obtained in Lemma~\ref{lem_discs};
		\item [$\bullet$] $D_3R(\bt,\phi,F(\bt,\phi))$ is an isomorphism on $\ockm{1}{\alp}(\bdisc;\Cn)$ for all $\bt\in\opar{t}$.
\end{itemize}

Now, fix a real-analytic $\phi\in \cont^3(K;\Cn)$ with $||\phi||_{\cont^3}<\eps_t$. Let $R_\phi:\cW_\phi\rightarrow \ockm{1}{\alp}{}(\bdisc;\Cn)$ be the map given by 
	\bes
		R_\phi(\bt,f)=R(\bt,\phi,f),
	\ees
where $\cW_\phi=\{(\bt,f)\in\rl^{n-1}\oplus\ockm{1}{\alp}(\bdisc;\Cn):(\bt,\phi,f)\in\cW\}$.
Note that $R_\phi(\bt,F(\bt,\phi))=0$ and $D_2R_\phi(\bt,F(\bt,\phi))\approx \id$, as long as $\bt\in\opar{t}$. Since $\phi$ is real analytic, $R_\phi$ is analytic on $\cW_\phi$. This follows from the analyticity of $E$ as shown in lemma \ref{lem_disc1d}, and the easily-checked fact that the map $f \rightarrow \phi(f)$ is analytic for $\phi$ analytic.

We apply the analytic implict function theorem for Banach spaces to conclude that for each $\bt\in D^{n-1}(t)$, there exist neighborhoods $U_\bt '\subset \opar{t}$ of $\bt$ and $W_\bt '\subset\ockm{1}{\alp}(\bdisc;\Cn)$ of $F(\bt,\phi)$, and an analytic map $F_{\phi,\bt}:U_\bt '\rightarrow W_\bt '$ such that $F_{\phi,\bt}(\bt)=F(\bt,\phi)$ and 
	\be\label{eq_uniq2}
		R_\phi(\bs,f)=0\ \text{for}\ (\bs,f)\in U_\bt '\times W_\bt '\
		 \text{if and only if}\ 
		f=F_{\phi,\bt}(\bs).
	\ee
As before, the $F_{\phi,\bt}$'s coincide when their domains overlap. Thus, there is an analytic map $F_\phi:\opar{t}\rightarrow\ockm{1}{\alp}(\bdisc;\Cn)$ such that $R_\phi(\bt,F_\phi(\bt))\equiv 0$ on $\opar{t}$. We set
\bes
	M'_{\text{TR}}
	=\left\{F_\phi(\bt)(\xi)	
	:(\xi,\bt)\in\cdisc\times D^{n-1}(t)\right\}.
\ees
The uniqueness in \eqref{eq_uniq1} and \eqref{eq_uniq2} shows that, in fact, $F_\phi(\cdot)=F(\cdot,\phi)$ and $M'_{\text{TR}}=M_{\text{TR}}$. Thus, we already know that $M'_{\text{TR}}$ is a $\cont^1$-smooth embedded manifold in $\Cn$. To show that $M'_{\text{TR}}$ is in fact a real-analytic manifold, it suffices to show that $\mathscr F:(\xi,\bt)\mapsto F_\phi(\bt)(\xi)$ is real-analytic on $\cdisc\times\opar{t}$.

Now, since $\mathscr F$ is the composition of $(\xi,\bt)\mapsto(\xi,F_\phi(\bt))$ and the map $\ev:(\xi,f)\mapsto \wt f(\xi)$, $\mathscr F\in\cont^\om(\Delta\times\opar{t})$; see Lemma~\ref{lem_ev}. To show that $\mathscr F\in\cont^\om(\cdisc\times\opar{t})$, we fix $\bt_0\in D^{n-1}(t)$.  Since $F_\phi$ is real-analytic, there is an $\eps>0$ such that for $\bt\in \bt_0 + \opar{\eps}$, $F_\phi(\bt)(\xi)=\sum_{\bet\in\mathbb{N}^{n-1}} h_\bet(\xi)(\bt - \bt_0)^\bet$ with $h_\bet\in\ockm{1}{\alp}(\bdisc;\C)$ and $||h_\bet||_{1,\alp}\lesssim r^{|\bet|}$ for some $r>0$. Without loss of generality, let $\bt=0$. Now, let $\xi_0\in\bdisc$ and $z_0=\mathscr{F}(\xi_0,0)$. Since $T=\{(z_1,...z_n)\in\Cn:z_1\in\bdisc, z_2,...,z_n\in\rl\}$ is a real-analytic totally real manifold in $\Cn$ there exists a biholomorphism $P$ near $\xi_0$ that maps an open piece of $T$ biholomorphically into $\Rn$ in $\Cn$, mapping $\xi_0$ to the origin. Similarly, there exists a biholomorphism $Q$ near $z_0$ that maps an open piece of $\sn_\phi$ biholomorphically into $\Rn$ in $\Cn$, mapping $z_0$ to the origin. Now, we let $Q^*(z_1,z')=Q(\sum_{\bet} h_\bet(z_1)	 (z')^\bet)$, where $z'=(z_2,...,z_n)$. From the analyticity of $F_\phi$, we have that $Q^*\in\hol(W)\cap\cont(W')$, where 
	\beas
	W&=&\{z_1\in\Delta:|z_1-\xi_0|<\eps\}\times \{z'\in\C^{n-1}:||z'||<\eps\},\\
	W'
	&=&\{z_1\in\cdisc:|z_1-\xi_0|\leq \eps\}\times \{z'\in\C^{n-1}:||z'||<\eps\}.
	\eeas
For $(z_1,...,z_n)$ close to $0$, we define 
	\bes
		P^*(z_1,z')=\begin{cases}
			Q^*\circ P^{-1}(z_1,z'),\ \ima z_1>0,\\
			\overline{Q^*\circ P^{-1}\overline{(z_1,z')}},\ \ima z_1<0.\\							
							\end{cases}
	\ees
Then, by the edge of the wedge theorem, $P^*$ extends holomorphically to a neighborhood of $(0,0)$ in $\Cn$, and thus, $\mathscr F$ extends analytically to a neighborhood of $\xi_0$ in $\cdisc\times\opar{t}$. Repeating this argument for every $\bt\in\opar{t}$, we obtain the real-analyticity of $M_{\text{TR}}$.  

\subsection{Regularity of the foliation for \texorpdfstring{$\boldsymbol{\cont^{2k+1}}$}{b}\ -smooth perturbations} 
\label{subsec_cksmooth}

In this section, we improve the regularity of the manifold $M_{\text{TR}}$ constructed in Theorem~\ref{thm_diskconst} under the assumption that the map $\Phi$ is $\cont^{2k+1}$-smooth, where $k\in\N$. Recall that for a fixed $t\in(0,1)$, Theorem~\ref{thm_diskconst} yields a neighborhood $N_t$ of $0$ in $\cont^1(K;\Cn)$ such that, for $\phi\in N_t$, $M_{\text{TR}}=\cF_\phi(\cdisc\times\opar{t})$ is a $\cont^1$-smooth submanifold in $\Cn$. Shrinking $N_t$ further, if necessary, we prove

\begin{theorem}\label{thm_cksmooth} For any $k\in\N$, $\phi\in N_t\cap \cont^{2k+1}(K;\Cn)$ and $\bt\in\opar{t}$, there exist neighborhoods $\cW_1$, $\cW_2\subset\opar{t}$ of $\bt$, and a $\cont^k$-smooth embedding $\cG_k:\cdisc\times \cW_1\rightarrow\Cn$ such that $\cG_k(\cdisc\times \cW_1)=\cF(\cdisc\times \cW_2)$. Thus, $M_{\text{TR}}$ is $\cont^k$-smooth. In particular, if $\phi\in N_t\cap \cont^\infty(K;\Cn)$, then $M_{\text{TR}}$ is a $\cont^\infty$-smooth manifold. 
\end{theorem} 

To establish Theorem~\ref{thm_cksmooth}, we first observe that when $\phi\in\cont^{2k+1}(K;\Cn)\cap N_t$, then $f_\bt:\xi\mapsto F(\bt,\phi)(\xi)$ is in $\ockm{2k}{\alp}(\bdisc;\Cn)$ (for every $0<\alp<1$) for every $\bt\in \opar{t}$. This follows from known regularity results for analytic discs attached to totally real manifolds in $\Cn$ (see \cite{Ch82}). Then, we will use the theory of partial indices --- introduced by Forstneri{\v c} for totally real manifolds in $\CC$, and generalized by Globevnik to higher dimensions (see \cite{Fo87} and \cite{Gl94}) --- to produce a $\cont^k$-smooth $(n-1)$-dimensional family of analytic discs attached to $\sn_\phi$ and show that these coincide with the ones that foliate $M_{\text{TR}}$ near $f_\bt$. Although, we provide all the necessary definitions below, we direct the reader to Sections~2-5 in \cite{Gl94} for more background on Hilbert boundary problems and partial indices. 

\begin{ntn}
In this section, we will sometimes express an $n\times n$ matrix over $\C$ as
	\bes
	\begin{pmatrix}
		a & \boldsymbol v\\
		\boldsymbol w^{\tr} & A
	\end{pmatrix},
	\ees
where $a\in\C$, $\boldsymbol v$, $\boldsymbol w\in \C^{n-1}$, and $A$ is an $(n-1)\times(n-1)$ matrix over $\C$.
\end{ntn}
Let $M$ be an $n$-dimensional totally real manifold in $\Cn$. Suppose $f:\cdisc \rightarrow\Cn$ is an analytic disc with boundary in $M$, i.e., $f\in\cont(\cdisc)\cap \hol(\Delta)$, and $f(\bdisc)\subset M$. Further, suppose $A:\bdisc\rightarrow \Gl(n;\C)$ is such that the real span of the columns of $A(\xi)$ is the tangent space $T_{f(\xi)}M$ to $M$ at $f(\xi)$, for each $\xi\in\bdisc$. Then, owing to the solvability of the Hilbert boundary problem for vector functions of class $\cont^\alp$ (see \cite[Sect. 3]{Gl94}, also see \cite{Ve67}), it is known that if $A$ is of class $\cont^\alp$ ($0<\alpha<1$), then there exist maps $F^+\!:\cdisc\rightarrow \Gl(n;\C)$
 and $F^-\!:\hat\C\setminus\Delta\rightarrow\Gl(n;\C)$,
 and integers $\kap_1\geq \cdots\geq  \kap_n$, such that
	\begin{itemize}
		\item $F^+\!\in\cont^\alp(\cdisc)\cap\hol(\Delta)$ and
				 $F^-\!\in\cont^\alp(\hat\C\setminus\Delta)
					\cap \hol(\hat\C\setminus\cdisc)$;
		\item for all $\xi\in\bdisc$,
			\be\label{eq_fact}
				A(\xi)\overline{A(\xi)^{-1}}=F^+(\xi)
				\begin{pmatrix} 
				\xi^{\kap_1} & 0  &\cdots & 0 \\
				0 & \xi^{\kap_2}  &\cdots & 0\\
				\vdots & \vdots & \ddots & \vdots\\
				0 & 0  & \cdots &\xi^{\kap_n} 
				\end{pmatrix}
				F^-(\xi),\quad  		
					\xi\in\bdisc.
			\ee
	\end{itemize}
Moreover, the integers $\kap_1\geq ...\geq \kap_n$ are the same for all factorizations of the type \eqref{eq_fact}. These integers are called the {\em partial indices of $M$ along $f$} and their sum is called the {\em total index of $M$ along $f$}. Using the factorization above, a normal form for the bundle $\{T_{f(\xi)}M:\xi\in\bdisc\}$ is obtained in \cite{Gl94}. In particular, it is shown that if the partial indices of $M$ along $f$ are even, then there is a $\cont^\alp$-map $\Theta:\cdisc\rightarrow\Gl(n;\C)$, holomorphic on $\Delta$, and such that for every $\xi\in\bdisc$, the real span of the columns of the matrix $\Theta(\xi)\Lam(\xi)$ is $T_{f(\xi)}M$, where $\Lam(\xi)=\operatorname{Diag}[\xi^{\kap_1/2},...,\xi^{\kap_n/2}]$. Conversely, suppose,  
\bea\label{eq_ind}
&\text{there is a $\Theta:\cdisc\rightarrow\Gl(n;\C)$ of class $\cont^\alp$, holomorphic on $\Delta$, such that}&\\
&\text{$\ima(A^{-1}\Theta\Lam)\equiv 0$ on $\bdisc$ or, equivalently, the real span of the columns of $\Theta(\xi)\Lam(\xi)$ is $T_{f(\xi)}M$}.&\notag
\eea
Then, for $\xi\in\bdisc$,
	\bes
		A(\xi)\overline{A(\xi)^{-1}}=\Theta(\xi)
				\begin{pmatrix} 
				\xi^{\kap_1} & 0  &\cdots & 0 \\
				0 & \xi^{\kap_2}  &\cdots & 0\\
				\vdots & \vdots & \ddots & \vdots\\
				0 & 0  & \cdots &\xi^{\kap_n} 
				\end{pmatrix}
				\overline{\Theta^{-1}(1/\overline{\xi})},
	\ees 
which, due to the holomorphicity of $\Theta$ on $\Delta$, is a factorization of type \eqref{eq_fact}. Thus, we obtain 

\begin{remark}\label{rem_indices} Suppose $f$ and $A$ are as above. Then, $A$ satisfies \eqref{eq_ind} if and only if $\kap_1,...,\kap_n$ are the partial indices of $M$ along $f$. Furthermore, if $A$ is of class $\ckm{k}{\alp}$, then $\Theta$ in \eqref{eq_ind} can be chosen to be of class $\ckm{k}{\alp}$. 
\end{remark}

A quick application of Remark~\ref{rem_indices} shows that the partial indices of $\sn$ along each $\fg_\bt$, $\bt\in  D^{n-1}(1)$, are $2,0,...,0$. This is because, for each fixed $\bt\in D^{n-1}(1)$ and $\xi\in\bdisc$, the real span of the columns of 
	\bes
		  \begin{pmatrix}
	   	i\sqrt{1-||\bt||^2}
				& 
				-\dfrac{\bt\xi}{\sqrt{1-||\bt||^2}}\\   
			\boldsymbol 0^{\tr}
				& 
					\id_{n-1}
    	\end{pmatrix}
		  \begin{pmatrix}
	   	\xi & \boldsymbol 0\\   
			\boldsymbol 0^{\tr} & \id_{n-1}
    	\end{pmatrix}
	\ees
is precisely $T_{\fg_\bt(\xi)}\sn$, and the factor on the left clearly extends to a holomorphic map (in $\xi$) from $\Delta$ to $\Gl(n;\C)$. We now use Remark~\ref{rem_indices} to establish a stability result for partial indices of $\sn_\phi$ along the disks constructed in Lemma~\ref{lem_discs}. 

\begin{lemma}\label{lem_stabind}
Let $\wt\Om$ and $F$ be as in Lemma~\ref{lem_discs}. Then, given any $t\in(0,1)$, there exists a neighborhood $N_t\subset\cont^3(K;\Cn)$ such that $\cpar{t}\times N_t\subset\wt\Om$, and for any $(\bt,\phi)\in \opar{t}\times N_t$, the partial indices of $\sn_\phi$ along $f_\bt:\xi\mapsto F(\bt,\phi)(\xi)$, $\xi\in\bdisc$, are $2,0,...,0$. 
\end{lemma}


\begin{proof} Let $t\in(0,1)$, and $N_t\subset\cont^3(K;\Cn)$ be as in Theorem~\ref{thm_diskconst}.  Recall that $\cF_\phi:(\xi,\bt)\mapsto F(\bt,\phi)(\xi)$ for $(\xi,\bt)\in \cdisc\times\opar{t}$. Note that $\cF_0(\xi,\bt)=(\sqrt{1-||\bt||^2}\,\xi,\bt)$ and
$D_{\xi,\bt}\cF_0(\xi,\bt)\in \Gl(n;\C)$ for all $(\xi,\bt)\in\cdisc\times \opar{t}$. 

Let $\eps>0$. As in the proof of Theorem~\ref{thm_diskconst}, $N_t$ can be chosen so that for each $\phi\in N_t$,
	\begin{enumerate}
		\item $\cF_{\text{bdy}}$ is a $\cont^1$-smooth parametrization of an open
		 totally real subset of $\sn_\phi$; where
		 \bes
		\cF_{\text{bdy}}:(\theta,\bt)\mapsto F(\bt,\phi)(e^{i\theta}), 
				\qquad (e^{i\theta},\bt)\in \bdisc\times\opar{t},
	\ees
		 
		\item $||D_{\xi,\bt}\cF_\phi-D_{\xi,\bt}\cF_0||_\infty<\eps$.
	\end{enumerate}
	
Now, we fix a $\phi \in N_t$ and let $\cF = \cF_\phi$. Since $\partl{}{\theta}=i\xi\partl{}{\xi}$ when $\xi=e^{i\theta}$, we have that 
	\be\label{eq_fact2}
		(D_{\theta,\bt}\cF_{\text{bdy}})(\theta,\bt)
			=\Theta_\bt(\xi)
		\begin{pmatrix}
	   	\xi & \boldsymbol 0\\   
			\boldsymbol 0^{\tr} & \id_{n-1}
    	\end{pmatrix}
				\quad \text{on}\ \bdisc,
	\ee
where 
	\bes
		\Theta_\bt(\xi)=(D_{\xi,\bt}\cF)(\xi,\bt)
				\begin{pmatrix}
	   	i & \boldsymbol 0\\   
			\boldsymbol 0^{\tr} & \id_{n-1}
    	\end{pmatrix}.
	\ees
Owing to $(1)$, the real span of the columns of the matrix $A_\bt(e^{i\theta})=(D_{\theta,\bt}\cF_{\text{bdy}})(\theta,\bt)$ is the tangent space to $\sn_\phi$ at $f_\bt(e^{i\theta})$. By $(2)$, if $\eps>0$ is sufficiently small, then $\Theta_\bt:\cdisc\rightarrow \Gl(n;\C)$. Thus, in order to apply Remark~\ref{rem_indices} to $f=f_\bt$ and $A=A_\bt$, we must show that $A$ is of class $\cont^\alp$, and $\Theta_\bt$ extends holomorphically to $\Delta$. We will, in fact, show that the entries of $(D_{\xi,\bt}\cF)(\cdot,\bt)$ are in $\ockm{1}{\alp}{}(\bdisc)$. First, since $\sn_\phi$ is $\cont^3$-smooth and $\xi\mapsto \cF(\xi,\bt)$ is an analytic disc attached to $\sn_\phi$, $\cF$ is $\cont^{2,\alp}$-smooth in $\xi$. This gives the $\ckm{1}{\alp}$-regularity of $D_\xi\cF(\cdot,\bt)$ on $\bdisc$. Next, note that $\cF(\xi,\bt)=\ev(\xi,F(\bt,\phi))$, where $\ev$ is the map defined in Lemma~\ref{lem_ev}. Thus, $D_{\bt}\cF(\cdot,\bt)(\bs)=D_\bt F(\bt,\phi)(\bs)(\cdot)$. Since $D_\bt F(\bt,\phi)$ is a bounded linear transformation from $\rl^{n-1}$ to $\ockm{1}{\alp}(\bdisc;\Cn)$, the entries of $D_\bt\cF(\cdot,\bt)$ are in $\ockm{1}{\alp}(\bdisc;\Cn)$, and therefore holomorphic in $\xi\in\Delta$. Thus, the indices of $\sn_\phi$ along $f_\bt$ are $2,0,...,0$. 
\end{proof}

 We will now employ the technique from \cite{Gl94} (also see \cite{Fo87}) to parametrize the family of all the analytic discs close to $f_\bt$ that are attached to $\sn_\phi$. Lemmas~\ref{lem_Glob6.1} and ~\ref{lem_Glob7.1} below are the $\ckm{k}{\alp}$-versions of the main results in Section~6 and ~7 of \cite{Gl94}. For the sake of brevity, we will keep some of the computations brief, and direct the reader to \cite{Gl94} for more details.  

For the rest of this section, we fix $t\in(0,1)$, $\phi\in N_t\cap \cont^{2k+1}(K;\Cn)$ ($k\geq 3$) and $\bt\in \opar{t}$. We let $\cM=M_{\text{TR}}$. Recall that by \cite{Ch82}, $f_\bt:\xi\mapsto \cF_\phi(\xi,\bt) $ is in 	$\ockm{2k}{\alp}{}(\bdisc;\Cn)$ and is a normalized analytic disc attached to $\sn_\phi$ (see Remark~\ref{rem_norm}). We fix a tubular neighborhood $\Om$ of $f_\bt(\bdisc)$ in $\Cn$ and a map $\rho^\phi:\Om\rightarrow \rl^n$ such that
	\begin{itemize}
		\item [$\triangleright$] $\rho^\phi=(\rho^\phi_1,...,\rho^\phi_n)
				\in\cont^{2k+1}(\Om;\rl^n)$;
		\item [$\triangleright$] $d\rho_1^\phi\wedge\cdots\wedge d\rho_n^\phi
						\neq 0$ on $\Om$; 
		\item [$\triangleright$]  $\sn_\phi\cap \Om=\{z\in \Om:\rho^\phi(z)=0\}$.
	\end{itemize} 
Let $X_1(\xi)=\partl{f_\bt}{\theta}(\xi)$. Since $\sn_\phi\cap \Om$ is $\cont^{2k+1}$-smooth and totally real, there exist $\cont^{2k}$-smooth maps $X_2,...,X_n:\bdisc\rightarrow\Cn$ such that for each $\xi\in\bdisc$, the real span of $X_1(\xi),...,X_n(\xi)$ is the tangent space to $\sn_\phi$ at $f_\bt(\xi)$. Given $p=(p_1,...,p_n)\in\ckm{k}{\alp}{}(\bdisc;\rl^n)$ and $q=(q_1,...,q_n)\in\ckm{k}{\alp}{}(\bdisc;\rl^n)$, let
	\bes
		\cE(p,q)=
		\sum_{j=1}^np_jX_j+
			i\sum_{j=1}^n(q_j+i\cH(q_j))X_j.
	\ees
Note that $\cE:\ckm{k}{\alp}{}(\bdisc;\rl^n)\times \ckm{k}{\alp}{}(\bdisc;\rl^n)\rightarrow \ckm{k}{\alp}{}(\bdisc;\Cn)$ is a linear isomorphism. This is because $X_j(\xi), iX_j(\xi)$, $1\leq j\leq n$, form a real basis of $\Cn$ and the standard Hilbert transform $\cH:\ckm{k}{\alp}{}(\bdisc;\rl^n)\rightarrow \ckm{k}{\alp}{}(\bdisc;\rl^n)$ is a bounded linear map.
\begin{lemma}\label{lem_Glob6.1} There exist neighborhoods $\cU_1$ of $0$ in $\ckm{k}{\alp}{}(\bdisc;\rl^n)$ and $\cU_2$ of $0$ in $\ckm{k}{\alp}{}(\bdisc;\C^n)$, and a $\cont^{k}$-smooth map $\cD:\cU_1\rightarrow \ckm{k}{\alp}{}(\bdisc;\C^n)$ such that 
	\begin{itemize}
		\item [$(i)$] for any $f\in \cU_2$, $f_\bt+f$ is attached to $\sn_\phi$ if and 
			only if $f=\cD(p)$ for some $p\in \cU_1$; and 
		\item [$(ii)$] there is an $\eta>0$ such that
			$||\cD(p)-\cD(p')||_{k,\alp}\geq \eta||p-p'||_{k,\alp}$
				 for all $p,p'\in \cU_1$.
	\end{itemize}
\end{lemma}
\begin{proof}
Let $\cU$ be a neighborhood of $0$ in $\ckm{k}{\alp}{}(\bdisc;\rl^n)$ such that, for all $p,q\in \cU$, $f_\bt(\xi)+\cE (p,q)(\xi)\in \Om$ for all $\xi\in\bdisc$. Consider the map 
	\bes
		\cR:(p,q)\mapsto \left(
			\xi\mapsto \rho^\phi\big(f_\bt(\xi)+\cE(p,q)(\xi)\big)\right)	
	\ees
on $\cU\times\cU$. Note that $\cR(0,0)=0$. By Lemma~5.1 in \cite{HiTa78}, $\cR:\cU\times\cU\rightarrow\ckm{k}{\alp}{}(\bdisc;\rl^n)$ is a $\cont^{k}$-smooth map. We claim that $(D_q\cR)(0,0):\ckm{k}{\alp}{}(\bdisc;\rl^n)\rightarrow \ckm{k}{\alp}{}(\bdisc;\rl^n)	$ is a linear isomorphism. This is because, for $h=(h_1,...,h_n)\in\ckm{k}{\alp}{}(\bdisc;\rl^n)$,
	\beas
		D_q\cR(0,0)(h)&=&
			\sum_{j=1}^n
				h_j\left<\nabla\rho^\phi_j(f_\bt),
					iX_k\right>_{\rl^{2n}}
			-\sum_{j=1}^n
				H(h_j)\left<\nabla\rho^\phi_j(f_\bt),
					X_k\right>_{\rl^{2n}}\\
			&=&
			\begin{pmatrix}
				\left<\nabla\rho^\phi_j(f_\bt),
					iX_k\right>_{\rl^{2n}}
			\end{pmatrix}
						\begin{pmatrix}
				h_1\\
				\vdots\\
				h_n
			\end{pmatrix}
			=C\begin{pmatrix}
				h_1\\
				\vdots\\
				h_n
			\end{pmatrix},
	\eeas
where $C$ is an $n\times n$ matrix with entries in $\ckm{k}{\alp}{}(\bdisc;\rl)$. Note that the second equality follows from the fact that $X_j(\xi)$ are tangential to $\sn_\phi$ at $f_\bt(\xi)$. It suffices to show the invertibility of $C$ at each $\xi\in\bdisc$. If, for some $\xi\in \bdisc$, $C(\xi)$ is not invertible, then there exist $a_1,...,a_n\in\rl$ such that $\sum_{j=1}^n a_jiX_j(\xi)$ is orthogonal to each $\nabla\rho^\phi_k(f_\bt(\xi))$ (as vectors in $\rl^{2n}$), which contradicts the total reality of $\sn_\phi$ at $f_\bt(\xi)$. Thus, by the implicit function theorem applied to $\cR$, there exist neighborhoods $\cU_1,\cU_1'\subseteq\cU$, and a $\cont^{k}$-smooth map $\cQ:\cU_1\rightarrow \cU_1'$ such that	
	\bes
		(p,q)\in \cU_1\times \cU_1'\ \text{satisfies}\ \cR(p,q)=0
			\iff 
		p\in\cU_1\ \text{and}\ q=\cQ(p).
	\ees
Now, setting $\cD(p)=\cE(p,\cQ(p))$, $\cU_2=\cE(\cU_1\times \cU_1')$, and recalling that $\cE$ is a linear isomorphism, we have $(i)$.

To establish $(ii)$, we note that $(D_p\cD)(0):\ckm{k}{\alp}{}(\bdisc;\rl^n)\rightarrow \ckm{k}{\alp}{}(\bdisc;\Cn)$ is the map
	\be\label{eq_derD}
		h\mapsto \sum_{k=1}^n h_jX_j.
	\ee
This computation uses the linearity of $D_q\cR(0,0)$; details can be found in \cite[Lemma~6.2]{Gl94}. Due to the nondegeneracy of the matrix $X=[X_1^{\tr},...,X_n^{\tr}]$, there exists an $\eta>0$ such that, for all $s\in U$ (after shrinking, if necessary), $(D_p\cD)(s)$ extends to a linear isomorphism $\mathcal I_s:\ckm{k}{\alp}{}(\bdisc;\Cn)\rightarrow \ckm{k}{\alp}{}(\bdisc;\Cn)$ satisfying $||\mathcal{I}_s(\cdot)||_{k,\alp}\geq \eta||\cdot||_{k,\alp}$ on $\ckm{k}{\alp}{}(\bdisc;\Cn)$. Assuming $\cU_1$ to be convex, we get $\cD(p')-\cD(p)=\left(\int_0^1 \mathcal I_{p+t(p'-p)}dt\right)(p'-p)$, and thus, 
	\bes
		||\cD(p')-\cD(p)||_{k,\alp}\geq \eta ||p'-p||_{k,\alp}
			\quad p,p'\in\ckm{k}{\alp}{}(\bdisc;\rl^n).
	\ees
\end{proof} 

The neighborhood $\cU_1$ obtained above parametrizes all the $\ckm{k}{\alp}$-discs close to $f_\bt$ that are attached to $\sn_\phi$. Next, we find those elements of $\cU_1$ that parametrize analytic discs attached to $\sn_\phi$. We direct the reader to Remark~\ref{rem_norm} for the definition of a normalized analytic disc. 

\begin{lemma}\label{lem_Glob7.1} There exists an open neighborhood $U$ of $0$ in $\rl^{n-1}$ and a $\cont^{k}$-smooth map $G:U\rightarrow \ockm{k}{\alp}{}(\bdisc;\Cn)$ such that 
	\begin{enumerate}
\item [$(a)$] $G(0)=0$;
\item [$(b)$] for each $\boldsymbol c\in U$, $f_\bt+G(\boldsymbol c)$ extends to a normalized analytic disc attached to $\sn_\phi$;
\item [$(c)$] for each neighborhood $V\subset U$ of $0$ in $\rl^{n-1}$, there is a $\tau_V>0$ so that if $f\in B_{k,\alp}(0;\tau_V)$ is such that $f_\bt+f$ is a normalized analytic disc attached to $\sn_\phi$, then $f=G(\boldsymbol c)$ for some $\boldsymbol c\in V$; 
\item [$(d)$] for each $\boldsymbol c_1,\boldsymbol c_2\in U$, $G(\boldsymbol c_1)\neq G(\boldsymbol c_2)$ if $\boldsymbol c_1\neq \boldsymbol c_2$. 
\item [$(e)$] the map $\cG:\cdisc\times U\rightarrow\Cn$ given by $(\xi,\boldsymbol c)\mapsto f_\bt+G(\boldsymbol c)$ is a $\cont^k$-smooth embedding. 
\end{enumerate}
\end{lemma}
\begin{proof} In Lemma~\ref{lem_stabind}, we proved that the indices of $\sn_\phi$ along $f_\bt$ are $2,0,...,0$. By Remark~\ref{rem_indices}, there is a map $\Theta=[\Theta_{j\ell}]_{1\leq j,\ell\leq n}\in\ockm{k}{\alp}{}(\bdisc;\Gl(n;\C))$ such that $X=\Theta Y$ on $\bdisc$, where 
	\bes
		Y(\xi)=
			\begin{pmatrix}
	   	\xi & \boldsymbol 0\\   
			\boldsymbol 0^{\tr} & \id_{n-1}
    	\end{pmatrix},\quad \xi\in\bdisc.
	\ees
Since $X_1={\bdy f_\bt}/\bdy\theta$, the above equation gives $(\bdy  f_\bt/{\bdy\theta})_1(\xi)=\xi\Theta_{11}(\xi)$. On the other hand, $(\bdy \fg_\bt/{\bdy\theta})_1(\xi)=i\xi\sqrt{1-||\bt||^2}$. Thus, shrinking $N_t$ in Theorem~\ref{thm_diskconst}, if necessary, we can make
	\bes
		||\Theta_{11}-i\sqrt{1-||\bt||^2}||_{\cont^\infty(\bdisc)}
		\leq ||f_\bt-\fg_\bt||_{\ckm{1}{\alp}(\bdisc)}
	\ees
small enough so that $\Theta_{11}(0)=\frac{1}{2\pi}\int_0^{2\pi}\Theta_{11}(e^{i\theta})d\theta \neq 0$. We work under this assumption for the rest of this proof.   

Now, let $\cU_1$, $\cU_2$ and $\cD$ be as in Lemma~\ref{lem_Glob6.1}. We determine the maps $f=f_\bt+\cD(p)$, $p\in \cU_1$, that extend holomorphically to $\Delta$. We have
	\beas
	 \cD(p)=\cE(p,\cQ(p))&=&
		\sum_{j=1}^n\left(p_j+i(\cQ_j(p)+i\cH\cQ_j(p))\right)X_j\\
		&=&
			\Theta\left( \sum_{j=1}^n p_jY_j+
			i\sum_{j=1}^n(\cQ_j(p)+i\cH\cQ_j(p))Y_j\right).
	\eeas
Note that $f_\bt$, $Y$ and $\cQ(p)+i\cH\cQ(p)$ extend holomorphically to $\Delta$. Moreover, $\Theta$ extends holomorphically to $\Delta$ with values in $\Gl(n;\C)$. Thus, $f=f_\bt+\cE(p,\cQ(p))$ extends holomorphically to $\Delta$ if and only if
	\be\label{eq_hol}
		\xi\mapsto \sum_{j=1}^np_j(\xi)Y_j(\xi)=(\xi p_1(\xi),p_2(\xi),...,p_n(\xi))
	\ee
extends holomorphically to $\Delta$. Let us assume that the map in \eqref{eq_hol} extends holomorphically to $\Delta$. Then, since $p_j$, $j=1,...,n$, are real-valued, we have that $p_j\equiv c_j$ for some real constants $c_2,...,c_n$. Moreover, $p_1(e^{i\theta})=\sum_{j\in\mathbb Z} a_j e^{ij\theta}$ for some $a_j\in\C$ satisfying  $a_0\in\rl$ and $a_j=\overline{a_{-j}}$, $j\in\N$. Thus, $\xi p_1(\xi)$ extends to a holomorphic map on $\Delta$ if and only if $a_j=0$ for all $|j|\geq 2$. Now, let $\mathfrak x =(p,q,r)\in\rl^3$ and $\boldsymbol c=(c_2,...,c_n)\in\rl^{n-1}$, and $\cP:\rl^{n+1}\mapsto \ckm{k}{\alp}{}(\bdisc;\rl^n)$ be the bounded linear map 
		\bes
			(\mathfrak x,\boldsymbol c)=(p,q,r,c_2,...,c_n)\mapsto 
				((p-iq)\overline{\xi}+r+(p+iq)\xi,c_2,...,c_n),
		\ees
then, based on the above argument,
\begin{center}
$(\bast)$ $f\in \cU_2$ extends holormorphically to $\Delta$ if and only if $f=\cD(\cP(\mathfrak x,\boldsymbol c))$ for some $(\mathfrak x,\boldsymbol c)\in \cP^{-1}(\cU_1)$.
\end{center}

Next, in order to reduce the dimension of the parameter space, we set $\mathfrak N=\wt \pi_{\text {ev}}\circ\cD\circ\cP$, where the map $\wt \pi_{\text {ev}}:\ockm{k}{\alp}(\bdisc;\Cn )\rightarrow \rl^3$ is given by $(f_1,...,f_n)\mapsto(\rea f_1(0),\ima f_1(0),\ima (f_1')(0))$. 
Then, $\mathfrak N:\cP^{-1}(\cU_1)\subset\rl^3\times\rl^{n-1}\rightarrow \rl^{3}$ is a $\cont^k$-smooth map with $\mathfrak N(0,0)=0$. We claim that $D_\mathfrak x\mathfrak N(0,0)$ is invertible. For this, using \eqref{eq_derD} and the fact that $X_1=\smpartl{f_\bt}{\theta}=i\xi\smpartl{f_\bt}{\xi}$, we note that
	\beas
	D_\mathfrak x\mathfrak N(0,0)(u,v,w)
&=& D\wt \pi_{\text {ev}}(0)\cdot D\cD(0)
	\left((u-iv)\overline{\xi}+w+(u+iv)\xi,0,...,0\right)\\
&=&D\wt \pi_{\text {ev}}(0)
	\left((u-iv)\tfrac{X_1(\xi)}{\xi}+wX_1(\xi)+(u+iv)\xi X_1(\xi)\right)\\
	&=&(au+bv,bu-av, Bu-Av+aw),
	\eeas
where $a=\rea(f_{t_1})'(0)$, $b=\ima(f_{t_1})'(0)$, $A=\rea(f_{t_1})''(0)$ and $B=\ima(f_{t_1})''(0)$. Here $f_{t_1}$ is the first component of the normalized analytic disc $f_\bt$. Thus, $\rea f_{t_1}'(0)>0$ and $D_\mathfrak x\mathfrak N(0,0)$ is invertible. We may, thus, apply the implicit function theorem to obtain neighborhoods $U$ of $0$ in $\rl^{n-1}$, $U'$ of $0$ in $\rl^3$, and a $\cont^k$-smooth map $\cA:W\rightarrow \rl^3$ such that $\mathfrak N(\mathfrak x,\boldsymbol c)=0$ for 
$(\mathfrak x,\boldsymbol c)\in U'\times U $ if and only if $\mathfrak x=\cA(\boldsymbol c)$. 

 Finally, we let $G:U\rightarrow \ockm{k}{\alp}{}(\bdisc;\Cn)$ be the map given by 
	\bes
		G:\boldsymbol c\mapsto 	\cD\big(\cP(\cA(\boldsymbol c),\boldsymbol c)\big).
	\ees  
It is clear that $G$ is $\cont^k$-smooth and $(a)$ holds. For $(b)$, we note that $(\pi_1\circ G)(\boldsymbol c)(0)=0$ for all $\boldsymbol c\in W$. Furthermore, by shrinking $U$ if necessary, we can ensure that $|(\pi_1\circ G)(\boldsymbol c)'(0)|<|(\pi_1\circ f_\bt)'(0)|$ for all $\boldsymbol c\in U$. Then, since $(\pi_1\circ f_\bt)'(0)>0$ and $\ima (\pi_1\circ G)(\boldsymbol c)'(0) =0$, we have that $\rea (\pi_1\circ G)(\boldsymbol c)'(0)>0$. Claim $(d)$ follows from Lemma~\ref{lem_Glob6.1}~(ii) and the fact that $\cP$ is injective. The argument for $(e)$ is similar to the proof of Theorem~\ref{thm_diskconst}. Now, for $(c)$, we let $V\subset U$ be a neighborhood of $0$ in $\rl^{n-1}$. Since $G$ is injective and continuous, $G(V)$ is open in $G(U)$ (in the subspace topology inherited from $\ckm{k}{\alp}(\bdisc;\Cn)$). Thus, there is an open set $\cV\subset \cU_2$ in $\ckm{k}{\alp}(\bdisc;\Cn)$ such that $G(V)=\cV\cap G(U)$, and so $G^{-1}(\cV)=V$. Thus, combining Lemma~\ref{lem_Glob6.1} $(i)$ and $(\bast)$, we have that, for $f\in \cV$, $f_\bt+f$ is an analytic disc attached to $\sn_\phi$ with $f_1(0)=0$ and $\ima f_1'(0)=0$ if and only if $f=G(c)$ for some $\boldsymbol c\in G^{-1}(\cV)=V$. To complete the proof of $(c)$, we choose $\tau_V>0$ so that $B_{k,\alp}(0;\tau_V)\subset\cV$. 
\end{proof}

\begin{remark}\label{rem_1alpversion}
We may repeat the proof of Lemma~\ref{lem_Glob7.1} in the $\ckm{1}{\alp}$-category to conclude that there exists an open neighborhood $U^*$ of $0$ in $\rl^{n-1}$ and a $\cont^{1}$-smooth injective map $G^*:U^*\rightarrow \ockm{1}{\alp}{}(\bdisc;\Cn)$ with $G^*(0)=0$ such that for each $\boldsymbol c\in U^*$, $f_\bt+G^*(\boldsymbol c)$ extends to a normalized analytic disc attached to $\sn_\phi$. Moreover, 
	\beas
	&(\dagger) \quad \text{for each neighborhood $V\subset U^*$ of $0$, there is a $\tau^*_V>0$ so that if $f\in B_{1,\alp}(0;\tau^*_V)$ and \qquad }& \\
	&\text{$f_\bt+f$ is a normalized analytic disc attached to $\sn_\phi$, then $f=G^*(\boldsymbol c)$ for some $\boldsymbol c\in V$},&
	\eeas
and $\cG^*:\cdisc\times U^*\rightarrow\Cn$ given by $(\xi,\boldsymbol c)\mapsto f_\bt+G(\boldsymbol c)$ is a $\cont^1$-smooth embedding. 
\end{remark}

\noindent {\em Proof of Theorem~\ref{thm_cksmooth}.} Recall that $\cM=M_{\text{TR}}$ is the manifold constructed in Theorem~\ref{thm_diskconst}. We let $\cM_k=\cG(\cdisc\times U)$ and $\cM_1=\cG^*(\cdisc\times U^*)$, where $\cG$ and $\cG^*$ are the maps defined in Lemma~\ref{lem_Glob7.1} $(e)$ and Remark~\ref{rem_1alpversion}, respectively. Note that $\cM$, $\cM_1$ and $\cM_k$ 	each contain the disc $f_\bt(\cdisc)$. To show that near $f_\bt(\cdisc)$, these three manifolds coincide, we will use the following proposition from \cite{Gl94}.

\begin{Prop}[{\cite[Prop.~8.1]{Gl94}}]
Let $X$ be a Banach space, $\om\subset\rl^n$ a neighborhood of $0$ and let $K$, $L:\om\rightarrow X$ be $\cont^1$-smooth maps such that $K(0)=L(0)$ and $(DK)(0)$, $(DL)(0)$ both have rank $n$. Suppose that for every neighborhood of $V\subset \om$ of $0$, there is a neighborhood $V_1\subset \om$ of $0$ such that $K(V_1)\subset L(V)$. Then, there are neighborhoods $\cV_1$ and $\cV_2$ of $0$ such that $K(\cV_1)=L(\cV_2)$.  
\end{Prop}

We first show that $\cM$ and $\cM_1$ coincide near $f_\bt(\cdisc)$. Shrinking $U^*$ if necessary, we may assume that $\bt+U^*\subset\opar{t}$. We set $\om=U^*\subset\rl^{n-1}$. For $\boldsymbol c\in \om$, we let $K(\boldsymbol c)=F(\bt+\boldsymbol c)$ and $L(\boldsymbol c)=f_\bt+G^*(\boldsymbol c)$, where $F$ and $G^*$ are the maps in Lemma~\ref{lem_discs} and Remark~\ref{rem_1alpversion}, respectively. Note that $K(0)=L(0)=f_\bt$ and $DK(0)$ and $DL(0)$ both have rank $n-1$. Now, let $V\subset \om$ be a neighborhood of $0$. We set $V_1=K^{-1}(B_{1,\alp}(f_\bt;\tau))$, where $\tau<\tau^*_V$ is sufficiently small so that $V_1\subset\om$. Then, for any $\boldsymbol c\in V_1$, $K(\boldsymbol c)$ is a normalized analytic disc attached to $\sn_\phi$ with property that $||K(\boldsymbol c)-f_\bt||_{1,\alp}<\tau<\tau^*_V$. Thus, by $(\dagger)$ in Remark~\ref{rem_1alpversion}, $K(\boldsymbol c)=f_\bt+L(\fd)$ for some $\fd\in V$. Thus, $K(V_1)\subset L(V)$. By the above proposition, there exist neighborhoods $\cV_1$, $\cV_2\subset\om$ of $0$ such that $K(\cV_1)=L(V_2)$. 
This shows that $\cM$ and $\cM_1$ coincide near $f_\bt(\cdisc)$. 

Next, we use the same approach to show that $\cM_1$ and $\cM_k$ coincide near $f_\bt(\cdisc)$.
In this case, we set $K(\boldsymbol c)=\iota\circ G(\boldsymbol c)$ and $L(\boldsymbol c)=G^*(\boldsymbol c)$, where $G^*$ and $G$ are the maps in Remark~\ref{rem_1alpversion} and Lemma~\ref{lem_Glob7.1}, respectively, and $\iota:\ckm{k}{\alp}(\bdisc;\Cn)\rightarrow\ckm{1}{\alp}(\bdisc;\Cn)$ is the inclusion map. Now, let $V\subset\om$ be a neighborhood of $0$. We set $V_1=G^{-1}(B_{k,\alp}(0;\tau))$, where $\tau<\min\{\tau_\om,\tau^*_V\}$ is sufficiently small so that $V_1\subset\om$. Then, for any $\boldsymbol c\in V_1$, $f_\bt+K(\boldsymbol c)$ is a normalized analytic disc attached to $\sn_\phi$ with property that $||K(\boldsymbol c)||_{1,\alp}<||G(\boldsymbol c)||_{k,\alp}<\tau<\tau^*_V$. Thus, by $(\dagger)$ in Remark~\ref{rem_1alpversion}, $f_\bt+K(\boldsymbol c)=f_\bt+L(\fd)$ for some $\fd\in V$. Thus, $K(V_1)\subset L(V)$. Once again, by the above proposition, there exist neighborhoods $\cV_1$, $\cV_2\subset\om$ of $0$ such that $K(\cV_1)=L(V_2)$. 
This shows that $\cM_k$ and $\cM_1$, and therefore $\cM_k$ and $\cM$, coincide near $f_\bt(\cdisc)$. This completes the proof of Theorem~\ref{thm_cksmooth}.
\qed

\section{Proof of parts \texorpdfstring{$(1)$}{c} to \texorpdfstring{$(5)$}{d} in Theorem~\ref{thm_main}}\label{sec_patch}
\subsection{Constructing $M$} So far, we have constructed that portion of the manifold $M$  whose leaves stay bounded away from $\Sing\sn_\phi$. We summarize the results from the previous sections as Theorem~\ref{thm_away} below. Note that we will use the following notation throughout this section. For $t\in(0,1)$,
	\beas	
				\sn_{\lessgtr t}&=&\{(z_1,x')\in\sn:||x'||\lessgtr t\},\\
		\bn_{\lessgtr t}&=&\{(z_1,x')\in\bn:||x'||\lessgtr t\}.
	\eeas
We also refer the reader to Section~\ref{subsec_notn} for the relationship between $\psi$, $\Psi$, $\phi=\operatorname{Inv}(\psi)$ and $\Phi$, and recall that $\sn_\phi=\Psi(\sn)$. Further, in view of Lemma~\ref{lem_sing}, if $\psi$ is sufficiently small, we may assume that $\Psi(\Sing\sn)=\Sing\sn_\phi$.

\begin{thmx}\label{thm_away} Let $k\geq 1$. Given $\de$ small enough, there is a $t\in(0,1)$ and an $\eps_t>0$ such that for all $\psi\in\cont^{2k+1}(\sn;\C)$ with $||\psi||_{\cont^3(\sn;\Cn)}<\eps_t$, there is a $\cont^{k}$-diffeomorphism $\varphiup:\bn_{<t}\rightarrow \Cn$ such that
\begin{itemize}
\item [$(i)$] $\varphiup(\sn_{<t})\subset\sn_\phi$, and for each $\bt\in D^{n-1}(t)$, $\Delta_\bt:=\varphiup\left(\{(z_1,x')\in\bn:x'=\bt\}\right)$ is an analytic disc attached to $\sn_\phi$.
\smallskip
\item [$(ii)$] $||\varphiup-\id||_{\cont^1(\bn_{<t})}<\de^2$.
\smallskip
\item  [$(iii)$] There exist $0<t_1<t<t_2<1$ such that $\Psi(\sn_{<t_1})\Subset \varphiup(\sn_{<t})\Subset\Psi(\sn_{<t_2})$.
\smallskip
\item [$(iv)$] There is a $t_3<t$ such that for $||\bt||\in(t_3,t)$, $\diam(\Delta_\bt)<7\de$ and $\sup_{z\in\Delta_\bt}\dist(z,\Sing\sn_\phi)<7\de$.
\end{itemize}
Moreover, $\varphiup$ has the same regularity as $\psi$, when $\psi$ is either $\cont^\infty$-smooth or real-analytic on $\sn$. 
\end{thmx}
\begin{proof} Let $\de\in(0,1)$ and $t=\sqrt{1-\de^2}$. Let $\eps_\eta>0$ be as in Lemma~\ref{lem_sing} for $\eta=\de^2$. Let $N_t\subset\cont^3(K;\Cn)$ be as in Theorem~\ref{thm_diskconst} (and Theorem~\ref{thm_cksmooth}). We choose $\eps(t)>0$ so that $||\psi||_{\cont^3(\sn;\Cn)}<\eps(t)$ implies that $\phi=\operatorname{Inv}(\psi)\in N_t$. Finally, we set $\eps_t=\min\{\eps_\eta,\eps(t),\de^2\}$. Then, $(i)$ and $(ii)$ follow from the construction in the previous section. 

For $(iii)$, we let $t_1=\sqrt{1-4\de^2}$. Note that $\varphiup(\sn_{<t_1})\Subset\varphiup(\sn_{<t})$ are connected open sets in $\sn_\phi$, and if $z\in\bdy \sn_{<t_1}$ and $w\in\bdy\sn_{<t}$,
	\beas
	||\varphiup(z)-\varphiup(w)||&\geq& ||z-w||-||\varphiup(z)-z||
											-||\varphiup(w)-w||\\
									&> & \frac{\de}{2}-\de^2-\de^2>2\de^2,
	\eeas
for sufficiently small $\de$. Thus, the ($2\de^2$)-neighborhood of $\varphiup(\sn_{<t_1})$ in $\sn_\phi$ is compactly contained in $\varphiup(\sn_{<t})$. But this neighborhood contains $\Psi(\sn_{<t_1})$ since $||\varphiup-\Psi||<2\de^2$. Thus, we have half of $(iii)$. For the second half of $(iii)$, we set $t_2=\sqrt{1-\de^2/4}$ and repeat a similar argument. 

For $(iv)$, we note that since $||\psi||_{\cont^3}<\eps_\eta$, we have that $||\Psi-\id||_{\cont^2(\sn)}<\de^2$ (see Lemma~\ref{lem_sing}). Hence, for $||\bt||\in\left(\sqrt{1-8\de^2},\sqrt{1-\de^2}\right)$, we have that for any $p,q\in\Delta_\bt$, 
	\beas
		||p-q||&\leq &||p-\varphiup^{-1}(p)||
			+||\varphiup^{-1}(p)-\varphiup^{-1}(q)||+||\varphiup^{-1}(q)-q||\\
				&\leq & \de^2+4\sqrt{2}\de+\de^2<7\de, 
	\eeas
for sufficiently small $\de$. A similar argument also gives the second part of $(iv)$.
\end{proof}

To construct $M$ near $\Sing\sn_\phi$, we will rely on the deep work of Kenig-Webster and Huang (see \cite{KeWe84} and \cite{Hu98}, respectively), where the local hull of holomorphy of an $n$-dimensional submanifold in $\Cn$ at a nondengenerate elliptic CR singularity is completely described. Although their results are local, the proofs in \cite{KeWe84} and \cite{Hu98} yield the following version of their result. Once again, we are using the compactness of $\Sing\sn_\phi$.  

\begin{thmx}[Kenig-Webster \cite{KeWe84}, Huang \cite{Hu98}]\label{thm_KWH} Let $k>> 8$ and $m_k=\lfloor\frac{k-1}{7}\rfloor$. There exist $\de_j>0$, $j=1,2,3$, and $\eps^*>0$ such that for any $\psi\in\cont^k(\sn;\Cn)$ with $||\psi||_{\cont^3(\sn;\Cn)}<\eps^*$, there is a $\cont^{m_k}$-smooth $(n+1)$-dimensional manifold $\widetilde M_{\de_1,\de_2}^\phi$ in $\Cn$ that contains some neighborhood of $\Sing\sn_\phi$ in $\sn_\phi$ as an open subset of its boundary and is such that
	\begin{itemize}
\item [$(a)$] Any analytic disc $f:\Delta\rightarrow\Cn$ that is smooth upto the boundary with $f(\bdisc)\subset\sn_\phi$, $\diam(f(\Delta))<\de_1$ and $\sup_{z\in f(\Delta)}\dist(z,\Sing\sn_\phi)<\de_2$, is a reparametrization of a leaf in $\widetilde M_{\de_1,\de_2}^\phi$.
\smallskip
\item [$(b)$] $\Psi\left(\{z\in\sn:\dist(z,\Sing\sn)<\de_3\}\right)\Subset \bdy\widetilde M_{\de_1,\de_2}^\phi$. Further, if $p\in\widetilde M_{\de_1,\de_2}^\phi$ is such that \linebreak $\dist(\Psi^{-1}(p),\Sing\sn)<\de_3$, then there is an embedded disk,  $f:\Delta\rightarrow\Cn$ (unique upto reparametrization) that is smooth upto the boundary, with $p\in f(\Delta)$, $f(\bdisc)\subset\sn_\phi$ and $f(\cdisc)\subset \widetilde M_{\de_1,\de_2}^\phi$, and the union of all such disks is a smooth $(n+1)$-dimensional submanifold, $\widetilde M_{\de_1,\de_2,\de_3'}^\phi$, of $\widetilde M_{\de_1,\de_2}^\phi$.
\smallskip
\item [$(c)$]If $\Pi$ is the projection map $(z_1,x'+iy')\mapsto (y')$ on $\Cn$, then $\left\lVert\Pi\Big|_{\widetilde M_{\de_1,\de_2}^\phi}\right\rVert_{\cont^1}\approx 0$. 
\end{itemize}
Moreover, $\widetilde M_{\de_1,\de_2}^\phi$ has the same regularity as $\psi$, when $\psi$ is either $\cont^\infty$-smooth or real-analytic on $\sn$. 
\end{thmx}

Now, given $\de_j$, $j=1,2,3$, and $\eps^*>0$ as in Theorem~\ref{thm_KWH}, we let $\de=\min\{\frac{\de_1}{7},\frac{\de_2}{7},\frac{\de_3}{3}\}$ and $\eps=\min\{\eps_t,\eps^*\}$, where $t>0$ and $\eps_t>0$ correspond to $\de$ as in Theorem~\ref{thm_away} (shrinking $\de$ further, if necessary). Then, for $\psi\in\cont^k(\sn;\Cn)$ with $||\psi||_{\cont^3(\sn;\Cn)}<\eps$, we let	 
	\bes
		M=\varphiup(\bn_{<t})\cup\widetilde M_{\de_1,\de_2,\de_3}^\phi.
	\ees
We now proceed to show that this indeed gives (up to an adjustment) the desired manifold. First, by Theorem~\ref{thm_away} $(iii)$ and Theorem~\ref{thm_KWH} $(b)$,
	\bes
		\sn_\phi=\Psi\left(\sn_{<\sqrt{1-4\de^2}}\right)
			\cup\Psi\left(\sn_{\geq \sqrt{1-4\de^2}}\right)\subset\bdy M\subseteq\sn_\phi.
	\ees
This follows from the fact that $\dist(z,\Sing\sn)\lesssim 2\de<\de_3$, when $z\in \sn_{\geq \sqrt{1-4\de^2}}$.  

Next, for the foliated structure and the regularity of $M$, we need only focus on $\varphiup(\bn_{<t})\cap\widetilde M_{\de_1,\de_2,\de_3}^\phi$. Let $p\in\varphiup(\bn_{<t})\cap\widetilde M_{\de_1,\de_2,\de_3}^\phi$. Then, $p=\varphiup(z_1,\bt)$ for some $(z_1,\bt)\in\bn_{<t}$, where recall that $t=\sqrt{1-\de^2}$. We first assume that $||\bt||>t_3=\sqrt{1-8\de^2}$. Then, by the choice of $\de$ and Theorem~\ref{thm_away} $(iv)$, $\diam(\Delta_\bt)<\de_1$, $\sup_{z\in\Delta_\bt}\dist(z,\Sing\sn_\phi)<\de_2$ and  $\dist(p,\Sing\sn_\phi)<\de_3$. Thus, by Theorem~\ref{thm_KWH} $(b)$, $\cdisc_\bt\subset\varphiup(\bn_{<t})\cap\widetilde M_{\de_1,\de_2,\de_3}^\phi$. By this argument, we see that the smooth $(n+1)$-dimensional manifold
	\bes
		B_{t_3,t}:=\bigcup_{t_3<||\bs||<t}\cdisc_\bt 
	\ees
lies in $\varphiup(\bn_{<t})\cap\widetilde M_{\de_1,\de_2,\de_3}^\phi$. Thus, $ M$ is a smooth manifold in a neighborhood of $p$. 

Next, suppose $p=\varphiup(z_1,\bt)\in \varphiup(\bn_{<t})\cap\widetilde M_{\de_1,\de_2,\de_3}^\phi$ is such that $||\bt||\leq \sqrt{1-8\de^2}$. We observe that the complement of $\bdy B_{t_3,t}$ in $\sn_\phi\cap\bdy\widetilde M_{\de_1,\de_2,\de_3}^\phi$ consists of two disjoint submanifolds of $\sn_\phi$ --- one, say $S_{\!\text I}$, containing $\Sing\sn_\phi$ and contained in a $(2\de)$-neighborhood of $\Sing\sn_\phi$, and another, say $S_{\!\text{II}}$, with the property that $\dist(S_{\!\text{II}},\Sing\sn_\phi)=2\sqrt{2}\de+O(\de^2)>2\de$. Since $p\in\widetilde M_{\de_1,\de_2,\de_3}^\phi$, it lies on some analytic disc $f(\Delta)$ attached to $\sn_\phi\cap\bdy M_{\de_1,\de_2,\de_3}^\phi$. By the uniqueness of these discs, $f(\bdy\Delta)$ cannot intersect $\bdy B_{t_3,t}$ because any disc whose boundary intersects $\bdy B_{t_3,t}$ lies completely in $B_{t_3,t}$ (as seen above), and $p\in f(\Delta)$ does not. Thus, either $f(\bdy\Delta)\subset S_{\!\text{I}}$ or $f(\bdisc)\subset S_{\!\text{II}}$ (as the two are disjoint). 
But since $S_{\!\text{I}}$ lies in the tubular $(2\de)$-neighborhood of $\Sing\sn_\phi$, which is a polynomially convex set, we must have that if $f(\bdisc)\subset S_{\!\text{I}}$, then $\dist(p,\Sing\sn_\phi)<2\de$. This contradicts the fact that $p=\varphi(z_1,\bt)$ with $||\bt||\leq \sqrt{1-8\de^2}$. Thus, $f(\bdisc)\subset S_{\!\text{II}}$. This, and the fact that 
	\bes
		 \Psi\left(\sn_{\geq \sqrt{1-4\de^2}}\right)\subset 
			S_{\!\text{I}}\cup\bdy B_{t_3,t}
	\ees
shows that if we shrink $\widetilde M_{\de_1,\de_2,\de_3}^\phi$ by removing $S_{\!\text{II}}$ and the discs attached to it, then 
	\bes
		M=\varphiup(\bn_{<t})\cup\widetilde M_{\de_1,\de_2,\de_3}^\phi
	\ees
is an $(n+1)$-dimensional manifold, as smooth as $\widetilde M_{\de_1,\de_2,\de_3}^\phi$, and is  foliated by analytic discs attached to its boundary $\sn_\phi$. Moreover, $ M$ is a $\cont^1$-small perturbation of $\bn$ in $\Cn$.

\subsection{$M$ as a graph}\label{subsec_graph} 
Let $\Pi:\Cn\rightarrow\Cn\times\rl^{n-1}$ be the map $(z_1,x'+iy')\mapsto (z_1,x')$. For $\psi\in\cont^k(\sn;\Cn)$ as above, we note that since $\sn_\phi$ is a $\cont^3$-small perturbation of $\sn$, we may write $\sn_\phi=\gr_{\bdy\Om}(h)$, where $\bdy\Om\subset\C\times\rl^{n-1}$ is a $\cont^k$-smooth $n$-dimensional manifold that is a $\cont^3$-small perturbation of $\sn$, and $h:\bdy\Om\rightarrow\rl^{n-1}$ is a $\cont^k$-smooth map that is $\cont^3$-close to zero. We make two observations. Since $\sn_\phi$ lies in the strongly pseudoconvex hypersurface $\bdy\Om\times i\rl^{n-1}$, $ M\subset \overline{\Om}\times i\rl^{n-1}$ with $\text{int}\: M\subset \Om\times i\rl^{n-1}$.
Next, since $T_p( M)$ at any $p\in M$ is a small perturbation of $T_{\Pi(p)}(\overline{\Om})$ (as manifolds with boundary in $\Cn$), $\Pi: M\rightarrow\overline{\Om}$ is a local diffeomorphism that restricts to a diffeomorphism between $\sn_\phi$ and $\bdy\Om$. Thus, $\Pi$ extends to a $\cont^{m_k}$-smooth diffeomorphism from $ M$ to $\overline{\Om}$, and we may write $ M=\gr{\left(H\right)}$ for some $\cont^1$-small $H:\overline{\Om}\rightarrow\rl^{n-1}$. Lastly, $H$ has the same regularity as $M$, when $M$ is either smooth or real-analytic. 

\section{Proof of $(6)$ in Theorem~\ref{thm_main}}
\label{sec_final}
\subsection{On the analytic extendability of $M$} In this section, we fix our attention on real-analytic perturbations of $\sn$. So far, we have: given $\de>0$, there is an $\eps>0$ so that for any any $\psi\in\cont^\om(\sn;\C)$ with $||\psi||_{\cont^3(\sn)}<\eps$, there is a $\cont^\om$-domain $\Om_\phi\subset\C\times\rl^{n-1}$, and a $\cont^\om$-map $H:\overline\Om_\phi\rightarrow\rl^{n-1}$, such that 
\begin{itemize}
\item [$\star$] $\bdy\Om_\phi$ and $H|_{\bdy\Om_\phi}$ are $\eps$-small perturbations (in $\cont^3$-norm) of $\sn$ and the zero map, respectively,
\item [$\star$] $M_\phi=\gr_{\Om_\phi}(H)$ is foliated by an $(n-1)$-parameter family of embedded analytic discs attached to $\sn_\phi$, and $||H||_{\cont^1(\overline\Om_\phi)}<\de$.
\end{itemize}

Now, assuming a lower bound on the radius of convergence of $\psi$, we establish the analytic extendability of $H$ (and therefore, $M$). Here, we identify $\psi\in\cont^\om(\sn;\C)$ with its complexification $\psi_\C$ on $\sn_\C$, where $\sn_\C=\{(z,\zbar)\in\C^{2n}:z\in\sn\}$. For $\rho>0$, we let $\cN_\rho\sn_\C=\{\xi\in\C^{2n}:\dist(\xi,\sn_\C)<\rho\}$ and $V_\rho(M_\phi)=\{z\in\Cn:\dist(z,M_\phi)<\rho\}$. 

\begin{Prop}\label{prop_graph} Given $\rho>0$, there is a $\rho '>0$ such that, 
 for every $\de>0$, there is an $\eps>0$ so that for $\psi_\C\in\hol(\cN_\rho\sn_\C)$, $\sup_{\overline{\cN_\rho\sn_\C}}||\psi_\C||<\eps$, there is a map $\fH:V_{\rho'}(M_\phi)\rightarrow\C^{n-1}$, holomorphic in $z_1$, $\zobar$ and $z'$, with $||\fH(z_1,\zobar,z')-z'||_{\cont^2}<\de$, 
such that $M_\phi$ is an open subset of $\{\zbar '=\fH(z_1,\zobar,z'):(z_1,z')\in V_{\rho'}(M_\phi)\}$.
\end{Prop}

Near $\Sing\sn_\phi$, this follows from the results in \cite{KeWe84} and \cite{Hu98}, where uniform analytic extendability of the local hulls of holomorphy past real-analytic nondegenerate elliptic points is established . Away from $\Sing\sn_\phi$, we obtain this by complexifying the construction of $M_{\text{TR}}$, and establishing a lower bound on the radius of convergence of its parametrizing map $\cF_\phi:\dom\rightarrow\Cn$  for every $\phi$ (or $\psi$) sufficiently small. We briefly elaborate on this below.  


In order to complexify the construction in Section~\ref{sec_away}, we need to expand our collection of function spaces. For $s\in(0,1)$, we set, $\Delta_s=(1+s)\Delta$ and $\ann_s=\{z\in\C:1-s<|z|<1+s\}$. We define $\ockm{1}{\alp}(\bdy\Delta_s)$ and $\ockm{1}{\alp}(\ann_s)$ in analogy with $\ockm{1}{\alp}(\bdy\Delta)$; see \eqref{eq_holhol}. For any open set $U\in\Cn$, we let $A(U)$ be the Banach spaces of continuous functions on $\overline{U}$, whose restrictions to $U$ are holomorphic. 
	\beas
	X^n(s)&=&\ockm{1}{\alp}({\bdy\Delta_s};\Cn)\times\ockm{1}{\alp}(\ann_s;\Cn),\\
	X^n_\rl(s)&=&\{(f,h)=(f_1,h_1,...,f_n,h_n)\in
	 X^n(s):h|_{\bdy\Delta}=\overline{f}|_{\bdy\Delta}\},\\
		Y^n(r)&=&A(\cN_r\sn_\C;\Cn),\\
%
		Y^{2n}_\rl(r)&=&\{(\varphi_1,...,\varphi_{2n})\in Y^{2n}(r):\varphi_2(z,\zbar)
		=\overline{\varphi_1}(z,\zbar),	\ima \phi_j(z,\zbar)=0, j=3,...,2n\},\\
		Z^n(r,s) &=&\{(\varphi,\eta,f,h)\in Y^{2n}(r)\times X^n(s)
		:(f,h)(\ann_s)\subset \cN_r\sn_\C\},\\
%
	\eeas
We need the bounded linear map $K_{r,s}: \rl\times Y^n(2r)\times\ockm{1}{\alp}(\bdy\Delta_{2s};\Cn)
		\rightarrow  \C\times Y^{2n}(r)\times  X^n(s)$ given by
	\beas
	&(x,\phi_1,...,\phi_n,f)\mapsto 
	(x+i0,\underbrace{\phi_1,\phi_1^*,(\rea\phi_2)^*,(\ima\phi_2)^*...,
		(\rea\phi_{n})^*,(\ima\phi_{n})^*}_{=:(\phi,\phi^*)},f,f^*),& 
	\eeas
where $\phi^*_1$, $(\rea \phi_j)^*$, $(\ima\phi_j)^*$ and $f^*$ are obtained by taking the holomorphic extensions of the real analytic functions $\overline{\phi_1}\big|_{\sn_\C}$, $(\rea\phi_j)\big|_{\sn_\C}$, $(\ima\phi_j)\big|_{\sn_\C}$, and $\overline{f}\big|_{\bdy\Delta}$, respectively. To keep the exposition short, we will only discuss the construction for the case $n=2$.

Now, fixing $r=\rho/2$ and $s=\rho/3$, and dropping all inessential references to $r$ and $s$,	we solve the following complexified version of \eqref{eq_attach} on $\ann_s$: given $\varphi\in Y^2$, find $(f,h)\in X^2$ satisfying
	\beas
		&(f_1-\varphi_1(f,h))(h_1-\varphi_2(f,h))
		+\left(\dfrac{f_2+h_2}{2}-\varphi_3(f,h)\right)^2=1&\\
		&f_2-h_2=\varphi_4(f,h),&
	\eeas  
so that $(f,h)\in X^2_\rl$ if $\varphi\in Y^2_\rl$. For this, we first define the following maps on $\C\times Z^2$.
	\beas
		\Sigma^\C&:&(\eta,\varphi,f,h)
		\mapsto\left(\eta+H_\C(\varphi_4(f,h))
			-\varphi_3(f,h)\right)^2,\ \text{and}\\ 
		P^\C&:&(\eta,\varphi,f,h)\mapsto 
		\big(\phi_1(f,h),\varphi_2(f,h),1-\Sigma(\varphi,\eta,f,h)\big),
	\eeas
where $H_\C:\ockm{1}{\alp}(\ann_s)\rightarrow\ockm{1}{\alp}(\ann_s)$ is the complexified Hilbert transform (see \cite{HiTa78}). We let $\Om^\C\subset A(\ann_r)^2\times A(\ann_r;\C\setminus(-\infty,0))$ be the domain of the operator $E^\C$ obtained by complexifying the map $E$ constructed in Lemma~\ref{lem_disc1d}. The range of $E^\C$ lies in $X^1$, and if  $(f,h)=E^\C(\varphi,\sigma)$, then 
\begin{itemize}
\item on $\ann_s$, $(f-\varphi_1)(h-\varphi_2)=\sigma$,
\item if $\varphi\in Y^2_\rl$ and $\sigma|_{\bdy\Delta}> 0$ , then $(f,h)=(E(\phi,\sqrt{\sigma}),\overline{E(\phi,\sqrt{\sigma})})$ on $\bdy\Delta$, i.e., $(f,h)\in X^2_\rl$,
\item for $c\in\C\setminus(-\infty,0]$, $E^\C(0,0,c)=(\sqrt{c}\xi,\sqrt{c}/\xi)$.
\end{itemize} 

Finally, we set $\cW^\C=\{\zeta\in \C\times Z^2: P^\C(\zeta)\in\Om^\C\}$, and define the map $R^\C:\cW^\C\rightarrow X^2$ as follows
\bes
	\zeta=(\eta,\varphi,f,h)\mapsto
		(f,h)-\big(E^\C\circ P^\C(\zeta),
		\eta+H_\C(\varphi_4(f,h))+i\varphi_4(f,h),
		\eta+H_\C(\varphi_4(f,h))-i\varphi_4(f,h)\big).
\ees
We note that $R^\C$ complexifies the map  $R^\rl:(\bt,\phi,f)\mapsto \pi\circ R^\C(\bt+i0,K(\phi,f))$, where  $\pi$ denotes the projection $(z_1,w_1,z_2,w_2)\mapsto (z_1,z_2)$, and $R^\rl=0$ gives equations \eqref{eq_attach} (attaching equation for $\sn_\phi$). Now, all the complexified maps constructed are holomorphic on their respective domains, and therefore, so is $R^\C$. Moreover, $(\eta,0,\mathfrak{f}_\eta,\mathfrak h_\eta)\in\cW^\C$, $R^\C(\eta,0,\fg_\eta)=0$ and $D_3R^\C(\eta,0,\fg_\eta)=\id$, for $\eta\in Q(1,s)=(-1,1)\times (-is,is)$, where $\mathfrak f_\eta(\xi)=(\sqrt{1-\eta^2}\xi,\eta)$ and $\mathfrak h_\eta(\xi)=(\sqrt{1-\eta^2}\xi^{-1},\eta)$. Thus, by repeating the argument in \S \ref{subsec_exist}, given $t_0<1-s$, $s_0<s$, there is an $\eps>0$ and a holomorphic map $F^\C:Q(t_0,s_0)\times\{\varphi\in Y^4:||\varphi||<\eps\} \rightarrow X^2$, such that $R^\C(\eta,\varphi,F^\C(\eta,\varphi))=0$. Now, setting $\cF^\C_\varphi:\Delta_s\times Q(t_0,s_0)\rightarrow\C^{2}$ by $(\xi,\eta)\mapsto\pi\circ F^\C(\eta,\varphi)(\xi)$, we have that 

\begin{itemize}
\item [$(a)$] $\varphi\mapsto\cF^\C_\varphi$ is continuous from $\{\varphi\in Y^4:||\varphi||<\eps\} $ to $\cont^2\left(\overline{\Delta_s\times Q(t_0,s_0)};\CC\right)$. Thus, $\cF^\C_\varphi$ is an embedding (for $\eps>0$ sufficiently small),
\item [$(b)$] if $\varphi=(\phi,\phi^*)\in Y^4_\rl$, then $\cF^\C_\varphi$ maps $\bdy\Delta\times(-t_0,t_0)$ onto an open set in $\sn$.
\end{itemize}
Now, to obtain Proposition~\ref{prop_graph} for $M_{\text{TR}}$, we apply the implicit function theorem on $\cF^\C_\varphiup(\Delta_s\times Q(t_0,s_0))$ to solve for $\overline w$ in terms of $z$, $\zbar$ and $w$.

\subsection{The polynomially convex hull of $\sn_\phi$ in the real-analytic case} We note that if $M$ is as constructed in the previous section, then due to its foliated structure, $M$ is contained in both the schlicht part of $\wt{\sn_\phi}$, and in $\wh{\sn_\phi}$. In this section, we show that when the perturbations are real-analytic and admit a uniform lower bound on their radii of convergence, then $M$ is in fact polynomially convex. This will complete the proof of Theorem~\ref{thm_main}. Our strategy is to globally `flatten' $M$,  which allows for $M$ to be expressed as the intersection of $n-1$ Levi-flat hypersurfaces, to each of which we can apply Lemma~\ref{lem_polcvxLF}. We note that when $n=2$, the flattening is unnecessary, and the final claim follows directly from Lemma~\ref{lem_polcvxLF} (as seen in Bedford's paper \cite{Be82}). 

\begin{lemma}\label{lem_flat} There is a neighborhood $\cW$ of $\overline\Om_\phi$ in $\Cn$ and a biholomorphism $G:\cW\rightarrow\Cn$ such that $M\Subset G(\cW)$ and $||G-\id||_{\cont^1}\lesssim \de$.
\end{lemma}

\begin{proof} 

We let $M'=\{(z_1,z')\in V_{\rho'/2}M_\phi:\zbar'=\fH(z_1,\zobar,z')\}$, where $\rho'$ and $\fH$ are as in Proposition~\ref{prop_graph}. since $M'$ is a small perturbation of $\gr(0)$ and is foliated by analytic discs, it admits a tangential $(1,0)$-vector field, $L=\smpartl{}{z_1}+a_2\smpartl{}{z_2}+\cdots a_n\smpartl{}{z_n}$, $a_2,...,a_n\in\cont^\om(M';\C)$, such that $[L,\overline L]\in\text{span}\{L,\overline L\}$ mod $HM'\otimes_\rl\C$ on $M'$. The conditions on $L$ give that 
\begin{itemize}
\item [$(a)$] $\overline L(\ba)\equiv 0$ on $M'$, i.e., $\ba$ is a CR-map on $M'$, where $\ba=(a_2,...,a_n)$, and
\item [$(b)$] $\ba(z_1,z')=\partl{\mathfrak H}{\zobar}(z_1,\zobar, z')$ along $M'$, since $L(\zbar '-\fH(z_1,\zobar,z'))=0$.
\end{itemize}
Thus, we get that $\ba$ extends as a holomorphic map, say $\bA$, to some neighborhood of $M'$. Since, $\fH$ (and, therefore $\ba$) has radius of convergence at least $\rho'/2$ on $M'$, $\bA$ is holomorphic on  $V_{\rho'/2}(M_\phi)$. Further, we have that $\bA(z_1,z')=\ba(z_1,\zobar, z',\fH(z_1,\zobar,z'))$ on $V_{\rho'/2}(M_\phi)$, which gives the bound $||\bA||_{\cont^1}\lesssim\de$ on $V_{\rho'/2}(M_\phi)$ (since $||a||_{\cont^1}<\de$ on $M'$, from $(b)$). 

We now construct the flattening map. By applying the implicit function theorem to the equation $\zbar'=\fH(z_1,\zobar,z')$ on $V_{\rho'/2}(M_\phi)$, we can solve for $y'$ in terms of $x_1$, $y_1$ and $z'$ to write $M'=\gr_{\Om'}{H}$, where $\Om'$ is the $(1+\rho'/2)$-tubular neighborhood of $\Om_\phi$ in $\C\times\rl^{n-1}$, and $H:\Om'\rightarrow\rl^{n-1}$ is a $\cont^\om$-map with $||H||_{\cont^1}\lesssim \de$. Shrinking $\eps$ further, we may assume that $\Om_\phi\subset B \subset\Om'$, where $B=(1+\rho'/4)\overline{\bn}$. Given $(z_1,x')\in B$, we let $w(z_1,x')=x'+iH(x_1,y_1,x')$. Now, on the metric space $\mathscr F=\{g\in\cont(B;\rl^{n-1}):\sup_{B}||g-w||<\rho'/2\}$, endowed with the sup-norm, we consider the map  
	\bes
	Q:g\mapsto (Qg)(z_1,x')= x'+iH(0,0,x')+\int_0^{z_1} A(\xi,g(\xi,x'))d\xi.
	\ees
To see this, note that for $g$, $g_1$, $g_2\in\mathscr F$, we have 
	\beas
	&\sup_B||Qg-w||\leq \sup_{B}||H(0,0,x')-H(x_1,y_1,x')||
		+\sup_{V_{\rho'/2}(M_\phi)}||A||\diam(B)
			\lesssim\de\left(1+\frac{\rho'}{4}\right),&\\
		&\sup_B||Qg_1-Qg_2||\leq
		 \sup_{V_{\rho'/2}(M_\phi)}||DA||\diam(B)\sup_B||g_1-g_2||
	\lesssim \de\left(1+\frac{\rho'}{4}\right)\sup_B||g_1-g_2||.&
	\eeas 
Shrinking $\eps>0$ further, if necessary, we can ensure that $\de(1+\rho'/4)<\min\{\rho'/2,1\}$. Thus, 	$Q(\mathscr F)\subset\mathscr{F}$, and $Q$ is a contraction, i.e., $||Qg_1-Qg_2||_{\mathscr F}<||g_1-g_2||_{\mathscr F}$, for all $g_1,g_2\in\mathscr F$. By the Banach fixed point theorem, there is a unique $g_0\in\mathscr F$ such that $Q(g_0)=g_0$. In other words, $G:(z_1,x')\mapsto(z_1,g_0(z_1,x'))$ is a solution of the flow equation
	\beas
		&\partl{g}{z_1}(z_1,x')
		=\left(1,\bA(z_1,g(z_1,x'))\right),&\quad \text{on}\ B,\\
			&g(0,x')=x'+iH(0,0,x'),&
				\quad \text{on}\ B_0=B\cap\{z_1=0\}.
	\eeas
By the local uniqueness and regularity of solutions to quasilinear PDEs with real-analytic Cauchy data, $G$ must  be real-analytic in $z_1$ and $x'$. Moreover, $||G-\id||_{\cont^1(B)}\lesssim\de$. Thus, $G$ extends to a biholomorphism in some neighborhood $\cW$ of $B$. Now, since $G_*(\bdy/ \bdy z_1)=L$ and $G(B_0)\subset M'$, by the uniqueness of integral curves,  $G(B) \subset M'$. Finally, if $z\in\bdy B$, then $||\Pi\circ G(z)-z||\lesssim\de$, where $\Pi:\Cn\rightarrow\C\times\rl^{n-1}$ is the projection map, and $\de$ can be made sufficiently small (by shrinking $\eps$) so that $\Om_\phi\subset (\Pi\circ G)(B)$, and thus, $M\Subset G(B)\subset G(\cW)$. This settles our claim. 
\end{proof}

Now, to complete the proof of the polynomial (and holomorphic convexity) of $M$, we need the following lemma. 

\begin{lemma}\label{lem_polcvxLF}
Let $D'\subset\C^{n-1}\times\rl$ be a domain containing the origin, and $F:D'\rightarrow\rl$ be a smooth function such that $\mathcal{L}'=\gr_{D'}(F)$ is a Levi-flat hypersurface. Then, for any strongly convex domain $D\Subset D'$ containing the origin, the set $\mathcal L=\gr_{\overline D}(F)$ is polynomially convex.
\end{lemma}
\begin{proof} We fix a $t_0\in(0,1)$ such that $D_t=(1+t)D\Subset D'$ for all $t\leq t_0$. Now, set $C=2(t_0+\sup_{\overline D_{t_0}}|F|)$. Since $\overline D_{t_0}\times [-iC,iC]$ is polynomially convex in $\Cn$, by a theorem due to Docquier and Grauert (see \cite{DoGr60}), it suffices to produce a family of pseudoconvex domains, $\{U_t\}_{0<t}$ in $D_{t_0}\times(-iC,iC)$ such that
\beas
	\overline{U_s}\subset U_t\ \text{if}\ s<t,
	\quad \bigcap_{s>t}\: \text{int}\: U_s=U_t,
	\quad \bigcup_{s<t}{U_s}=U_t,
	\quad \mathcal{L}=\bigcap_{0<t} U_t,\
	\text{\hspace{3pt}and\hspace{3pt} }  
	D_{t_0}\times(-iC,iC)=\bigcup_{0<t} U_t.
\eeas
 We use the notation $(z^*,w)$ to denote a point in $\C^{n-1}\times\C$, with $w=u+iv$. Now, consider the following pseudoconvex domains. 
\beas
	U_{t}=
	\begin{cases}
	\{(z^*,w):(z^*,u)\in D_t, |v-F(z^*,u)|<t\},&\ 
\quad 0<t\leq t_0,\\
	\{(z^*,w):(z^*,u)\in D_{t_0},\max(-C,F(z^*,u)-t) <v<\min(F(z^*,u)+t,C)\},&\ 
\quad t>t_0.
	\end{cases}
	\eeas
The claim now follows. 
\end{proof}

Finally, given $j=2,...,n$, let $Y_j$ denote the hyperplane $\{z\in\Cn:\ima z_j=0\}$. We set 
	\bes
		\mathcal{L}_j'=G\left(\cW\cap Y_j\right)
	\ees
Shrinking $\eps$ further, if necessary, we have that $\mathcal{L}_j'$ is a graph of some smooth function $F^j$ over some open set $D_j'\subset Y_j\cong \C^{n-1}\times\rl$ such that $\Om_\phi\Subset D_j'\subset \cW\cap Y_j$. We now choose a strongly convex domain $\cE\subset\Cn$ such that 
\begin{itemize}
\item [$*$] $\cE\cap Y_j\Subset D_j'$, and 
\item [$*$] $\cE\cap Y_2\cap\cdots\cap Y_n=\Om_\phi$.
\end{itemize}
This can be obtained, for instance, by letting $\mathcal{E}=\{\tau^2 p(z,x')+||y'||^2<0\}$, where $p$ is a smooth strongly convex exhaustion function of $\Om_\phi$ with $p\geq -1$ (see \cite{Bl97}), and $\tau>0$ is small enough. Now, we apply Lemma~\ref{lem_polcvxLF} to $D_j'$, $F^j$ and $D_j=\cE\cap Y_j$, and conclude that $\mathcal{L}_j=\gr_{\overline{\cE_j}}(F^j)$ is polynomially convex. However, 
	\bes
		M=\bigcap_{j=2}^n \mathcal{L}_j.
	\ees
Thus, $M$ is polynomially convex. 

\bibliography{refs}{}
\bibliographystyle{plain}

\end{document}